\documentclass[11pt]{amsart}

\usepackage{amssymb, amsmath}

 \newtheorem{thm}{Theorem}[section]
 \newtheorem{cor}[thm]{Corollary}
 \newtheorem{lem}[thm]{Lemma}
 \newtheorem{prop}[thm]{Proposition}
 \theoremstyle{definition}
 \newtheorem{defn}[thm]{Definition}
 \newtheorem{notn}[thm]{Notation}
 \theoremstyle{remark}
 \newtheorem{rem}[thm]{Remark}
\numberwithin{equation}{section}
\DeclareMathOperator{\Pspec}{P.Spec}
\DeclareMathOperator{\spec}{Spec}

\DeclareMathOperator{\Pprim}{P.Prim}

\begin{document}

\title[Poisson Hopf algebra]{A Poisson Hopf algebra related to a twisted quantum group}

\author{Sei-Qwon Oh}

\address{Department of Mathematics, Chungnam National  University, 99 Daehak-ro,   Yuseong-gu, Daejeon 34134, Korea}

\email{sqoh@cnu.ac.kr}

\thanks{This work is supported by National  Research Foundation of Korea Grant 2012-007347. The author thanks the Korea Institute for Advanced Study for the warm hospitality during a part of the preparation of this paper.}



\date{September 12, 2015.}


\begin{abstract}
A Poisson algebra $\Bbb C[G]$ considered as a Poisson version  of the twisted quantized coordinate ring $\Bbb C_{q,p}[G]$, constructed by Hodges, Levasseur and Toro in \cite{HoLeT}, is obtained
and its Poisson structure is investigated.
This establishes that all Poisson prime and primitive ideals of $\Bbb C[G]$ are characterized.
Further it is shown that $\Bbb C[G]$ satisfies the Poisson Dixmier-Moeglin equivalence and that
Zariski topology on the space of Poisson primitive ideals of $\Bbb C[G]$ agrees with the quotient
topology induced by the natural surjection from the maximal ideal space of $\Bbb C[G]$ onto the Poisson
primitive ideal space.
\end{abstract}

\maketitle

\section*{Introduction}
There are many evidences that Poisson structures of Poisson algebras are heavily related to algebraic structures of their quantized algebras. Hodges and Levasseur \cite{HoLe}, Hodges, Levasseur and Toro \cite{HoLeT} and Joseph \cite{Jos} constructed bijections between the primitive
ideal space of the quantized coordinate ring of a semisimple Lie group $G$ and the set of
symplectic leaves in $G$ with a Poisson bracket arising from the quantization process.
In this case symplectic leaves correspond to the Poisson primitive ideals in the coordinate ring of $G$.
In \cite{Oh7} and \cite{Oh12}, the author constructed  the Poisson symplectic algebra and the multiparameter Poisson Weyl algebra
such that their Poisson spectra are homeomorphic to the spectra of the quantized symplectic algebra and the multiparameter quantized Weyl algebra, respectively. Moreover, in \cite{Good3}, \cite{GoLa}, \cite{GoLet2}, \cite{GoLet5}, \cite{OhPaSh1},   \cite[\S 2]{Good4} and
\cite[\S 3.3]{HoKa}, we  find  that Poisson structures of Poisson algebras are almost the same as the algebraic structures of their quantized algebras.

Let $\frak g$ be a finite dimensional complex semi-simple Lie algebra associated with a connected semisimple Lie group $G$ and let $r=\sum_{\alpha\in{\bf R}^+}x_\alpha\wedge x_{-\alpha}\in \bigwedge^2\frak g$, where $x_\alpha$ are
root vectors of $\frak g$ such that $(x_\alpha|x_{-\beta})=\delta_{\alpha\beta}$. Then $r$ is called the standard classical $r$-matrix,  the quantized universal enveloping algebra $U_q(\frak g)$ is a quantization of the enveloping algebra $U(\frak g)$ by $r$ and
the quantized function algebra $\Bbb C_q[G]$ is obtained as an algebra of functions on $U_q(\frak g)$. Let $\frak b^{\pm}$ be the Borel subalgebras of $\frak g$. Then the pair of  subalgebras  $U_q(\frak b^+)$ and  $U_q(\frak b^-)$ of $U_q(\frak g)$ is a dual pair of Hopf algebras.
Let $u$ be an alternating form on the dual space $\frak h^*$, where $\frak h$ is a  Cartan subalgebra of $\frak g$.  Hodges,  Levasseur and Toro \cite{HoLeT}  constructed   a dual pair of Hopf algebras $U_{q,p^{-1}}(\frak b^+)$ and $U_{q,p^{-1}}(\frak b^-)$ by twisting that pair in the case when $u$ is algebraic and obtained the  Drinfeld double $D_{q,p^{-1}}(\frak g)=U_{q,p^{-1}}(\frak b^+)\bowtie U_{q,p^{-1}}(\frak b^-)$. Finally they found a Hopf algebra  $\Bbb C_{q,p}[G]$, called the multiparameter quantized coordinate ring, as an algebra of  functions on $D_{q,p^{-1}}(\frak g)$, and proved that the prime and primitive ideals of $\Bbb C_{q,p}[G]$ are indexed by the elements of the double Weyl group $W\times W$ using a natural action of the torus $H$ associated with $\frak h$ and the adjoint action of a Hopf algebra.
From this  Hopf algebra $\Bbb C_{q,p}[G]$ arises the question: Is there a Poisson Hopf algebra related to $\Bbb C_{q,p}[G]$? A main aim of this paper is to give a solution for this question. Here we find a Poisson Hopf algebra $\Bbb C[G]$ considered as a Poisson version of $\Bbb C_{q,p}[G]$ and  establish that the Poisson structure of $\Bbb C[G]$ is an analogue to the algebraic structure of $\Bbb C_{q,p}[G]$.

In the section 1, we construct a Lie algebra $\frak d$ such that the enveloping algebra $U(\frak d)$ may be considered as a classical algebra of $D_{q,p^{-1}}(\frak g)$.
 In the section 2, we prove that the standard $r$-matrix $r$ still makes $\frak d$ a Lie bialgebra, find  the Poisson bracket $\{\cdot,\cdot\}_r$ on an algebra $\Bbb C[G]$ of functions  on $U(\frak d)$ and get finally a Poisson bracket $\{\cdot,\cdot\}$ on $\Bbb C[G]$ by twisting $\{\cdot,\cdot\}_r$ using  a skew symmetric bilinear form $u$ on $\frak h^*$,
 that makes $\Bbb C[G]$ a Poisson Hopf algebra.
In the section 3, we establish that  the Poisson prime ideals of $\Bbb C[G]$ are indexed by the elements of the double Weyl group $W\times W$. In the section 4, we define the Poisson adjoint action of a Poisson Hopf algebra and prove that the Poisson central elements are the fixed elements under the Poisson adjoint action. Finally, in the section 5, we show that the Poisson structure
of $\Bbb C[G]$ is an analogue to the algebraic structure of $\Bbb C_{q,p}[G]$ by using the  Poisson adjoint action defined in the section 4. Moreover we prove that  $\Bbb C[G]$ satisfies  the Poisson Dixmier-Moeglin equivalence and that  the Poisson primitive ideal space of $\Bbb C[G]$ is a quotient space of  its classical space.

Several parts of the paper are modified  from  those of \cite{HoLeT} by using Poisson terminologies because $\Bbb C[G]$ is a good Poisson analogue of the quantum group $\Bbb C_{q,p}[G]$ in \cite{HoLeT}, all fields are of characteristic zero and vector spaces are over the complex number field $\Bbb C$ unless stated otherwise. Moreover all Poisson algebras considered here are commutative.
\medskip

  A Poisson ideal $P$ of $A$ is said to be  {\it Poisson prime} if, for all Poisson ideals $I$ and $J$, $IJ\subseteq P$ implies $I\subseteq P$ or $J\subseteq P$.  Note that if $A$ is noetherian  then a Poisson prime ideal of $A$ is a prime ideal by
  \cite[Lemma 1.1(d)]{Good3}.  A Poisson ideal $P$  is said to be  {\it Poisson primitive} if there exists a maximal ideal $M$ such that $P$ is the largest Poisson ideal contained in $M$.
Note that Poisson primitive is Poisson prime.


\section{Algebra of functions}

\subsection{Notation}
Here we recall well-known facts in \cite{Kac}, which are summarized in \cite[Chapter 2]{HoKa}. Let $C=(a_{ij})_{n\times n}$ be an indecomposable and  symmetrizable generalized Cartan matrix  of finite type.  Hence there exists positive integers
$\{d_i\}_{1\leq i\leq n}$  such that the matrix $DC$ is symmetric positive definite, where  $D=\text{diag}(d_i)$ is   the diagonal matrix. (In \cite[Chapter 2]{HoKa}, each $d_i$ is denoted by $s_i$.)
Let $\frak g=(\frak g, [\cdot,\cdot]_{\frak g})$ be the finite dimensional Lie algebra over
the complex number field $\Bbb C$ associated to  $C=(a_{ij})_{n\times n}$.

Let  $\frak h$  be a Cartan subalgebra of $\frak g$ with simple roots $\alpha_1,\ldots,
\alpha_n$, ${\bf R}$ the root system, ${\bf R}^+$ the set of positive roots and $W$ the Weyl group.  Choose $h_i\in\frak h$, $1\leq i\leq n$, such that
\begin{equation}\label{Z}
\alpha_j:\frak h\longrightarrow \Bbb C,\ \ \alpha_j(h_i)=a_{ij}\ \text{ for all }j=1,\ldots,n.
\end{equation}
Then $\{h_i\}_{i=1}^n$ forms  a basis of $\frak h$ since  $C$ has rank $n$, and
$\frak g$ is generated by $h_i$ and $x_{\pm\alpha_i}$, $i=1, \ldots , n$, with relations
$$\begin{array}{c}
[h_i,h_j]_{\frak g}=0,\  [h_i,x_{\pm\alpha_j}]_{\frak g}=\pm a_{ij}x_{\pm\alpha_j}, \ [x_{\alpha_i},x_{-\alpha_j}]_{\frak g}=\delta_{ij}h_i,\\
(\text{ad}_{x_{\pm\alpha_i}})^{1-a_{ij}}(x_{\pm\alpha_j})=0,\ \ i\neq j
\end{array}$$
by \cite[Definition 2.1.3]{HoKa}. (In \cite[Definition 2.1.3]{HoKa}, $x_{\alpha_i}$ and $x_{-\alpha_i}$ are denoted by $e_i$ and $f_i$ respectively.)
Denote by $\frak n^+$ and $\frak n^-$ the subspaces of $\frak g$ spanned by root vectors with positive and negative roots respectively and set $\frak n=\frak n^+\oplus\frak n^-$. Hence
$$\frak g=\frak h\oplus\frak n=\frak n^+\oplus\frak h\oplus\frak n^-.$$
For each $\beta\in\bf{R}$, we will write $x_\beta$ for a root vector with root $\beta$, hence $\frak g_\beta=\Bbb Cx_\beta\subseteq\frak n$, where   $\frak g_\beta$ is the root space of $\frak g$ with root $\beta$.

By \cite[\S2.3]{HoKa}, there exists a nondegenerate  symmetric bilinear form $(\cdot|\cdot)$ on $\frak h^*$ given
by
\begin{equation}\label{Y}(\alpha_i|\alpha_j)=d_ia_{ij}
\end{equation}
 for all $i,j=1,\ldots,n$.
This  form $(\cdot|\cdot)$ induces the isomorphism
$\frak h^*\longrightarrow\frak h, \lambda\mapsto h_\lambda$, where $h_\lambda$ is  defined by
$$(\alpha_i|\lambda)=\alpha_i(h_\lambda)\ \ \text{ for all } i=1,\ldots,n.$$
(The isomorphism $\frak h^*\longrightarrow\frak h,  \lambda\mapsto h_\lambda$, is denoted by $\nu^{-1}$ in \cite[\S2.3]{HoKa}.)
Note that, by (\ref{Z}) and (\ref{Y}),
$$h_{\alpha_i}=d_ih_i,\ \ i=1,\ldots, n.$$
Identifying $\frak h^*$ to $\frak h$   via $\lambda\mapsto h_\lambda$,
 $\frak h$ has a nondegenerate  symmetric bilinear form $(\cdot|\cdot)$ given by
 $$(\lambda| \mu)=(h_\lambda|h_\mu)=\lambda(h_\mu).$$
 For example, $(h_i|h_j)=(d_i^{-1}h_{\alpha_i}|d_j^{-1}h_{\alpha_j})=d_i^{-1}d_j^{-1}(\alpha_i|\alpha_j)=d_j^{-1}a_{ij}.$
This is extended to
a nondegenerate $\frak g$-invariant symmetric bilinear form on $\frak g$ by  \cite[(2.7) and Proposition 2.3.6]{HoKa}
and \cite[Theorem 2.2 and its proof]{Kac}:
\begin{equation}\label{YA}
(h_i|h_j)=d_j^{-1}a_{ij},\ \ (h|x_{\alpha})=0,\ \ (x_{\alpha}|x_{\beta})=0\text{ if }\alpha+\beta\neq0,\ \
(x_{\alpha_i}|x_{-\alpha_j})=d_i^{-1}\delta_{ij}
\end{equation}
 for $h\in\frak h$ and $\alpha,\beta\in{\bf R}$.

\begin{lem}
For each $\alpha\in\bf{R}^+$,
\begin{equation}\label{X}
[x_{\alpha},x_{-\alpha}]_{\frak g}=(x_\alpha|x_{-\alpha})h_{\alpha}.
\end{equation}
\end{lem}

\begin{proof}
By \cite[Proposition 2.3.6]{HoKa},
$$(h_\lambda|[x_{\alpha},x_{-\alpha}]_{\frak g})=([h_\lambda,x_\alpha]_{\frak g}|x_{-\alpha})=(\lambda|\alpha)(x_\alpha|x_{-\alpha})=(x_\alpha|x_{-\alpha})(h_\lambda|h_\alpha)
=(h_\lambda|(x_\alpha|x_{-\alpha})h_\alpha)$$
for each $h_\lambda\in\frak h$. Hence the result holds.
\end{proof}

\subsection{}\label{SKEW}
 Let $u\in\bigwedge^2\frak h$. We may write
$$u=\sum_{1\leq i,j\leq n}u_{ij} h_i\otimes h_j$$
for some skew symmetric matrix $(u_{ij})$ since $\{h_i\}_{i=1}^n$ forms a basis of $\frak h$. The element $u$ can be considered  as a skew symmetric (alternating)  form
on $\frak h^*$ by
$$u(\lambda,\mu)=\sum u_{ij} \lambda(h_i)\mu(h_j)$$
for any $\lambda,\mu\in \frak h^*$.
Hence there exists a unique linear map  $\Phi:\frak h^*\longrightarrow \frak h^*$ such that
$$u(\lambda,\mu)=(\Phi(\lambda)|\mu)$$
for any $\lambda,\mu\in \frak h^*$ since the form $(\cdot|\cdot)$ on $\frak h^*$ is nondegenerate. Set
$$ \Phi_+=\Phi + I,\ \ \  \Phi_-=\Phi -I,$$
where $I$ is the identity map on $\frak h^*$.  Thus
\begin{equation}\label{Phi}
\begin{array}{c}
(\Phi_+\lambda|\mu)=u(\lambda,\mu)+(\lambda|\mu)\\
(\Phi_-\lambda|\mu)=u(\lambda,\mu)-(\lambda|\mu)
\end{array}
\end{equation}
 for all $\lambda,\mu\in\frak h^*$.

\subsection{}\label{var}
Fix a vector space $\frak k$ isomorphic to $\frak h$ and let  $\varphi:\frak h\longrightarrow\frak k$ be an isomorphism of vector spaces. For each $\lambda\in\frak h^*$, denote by $k_\lambda\in\frak k$ the  element $\varphi(h_\lambda)$. Let
$$\frak g'=\frak k\oplus\frak n$$
and define a skew symmetric bilinear product $[\cdot,\cdot]_{\frak g'}$ on $\frak g'$ by
\begin{equation}\label{Gprime}
\begin{array}{ll}
\ [k_\lambda,k_\mu]_{\frak g'}=0, &
[k_\lambda, x_\alpha]_{\frak g'}=(\alpha|\lambda)x_\alpha, \\
 \ [x_\alpha,x_\beta]_{\frak g'}=[x_\alpha,x_\beta]_{\frak g}\ \  (\alpha\neq-\beta),&
[x_\alpha,x_{-\alpha}]_{\frak g'}=\varphi([x_\alpha,x_{-\alpha}]_{\frak g})
\end{array}
\end{equation}
for all $\lambda,\mu\in\frak h^*$ and $x_\alpha,x_{-\alpha},x_\beta\in\frak n$. Since
$[k_\lambda, x_\alpha]_{\frak g'} =[h_\lambda,x_\alpha]_{\frak g}$  for each $\lambda\in\frak h^*$,  $\frak g'$ is the Lie algebra isomorphic to $\frak g$ such that each element $k_\lambda\in\frak k$ corresponds to the element  $h_\lambda$ of $\frak g$. That is, $\frak g'$ is the Lie algebra $\frak g$ such that the elements $h_\lambda\in\frak h$ are replaced by $k_\lambda\in\frak k$.
Note that
\begin{equation}\label{C}
[x_\alpha,x_\beta]_{\frak g'}=[x_\alpha,x_\beta]_{\frak g}\in\frak n
\end{equation}
for all $\alpha,\beta\in\bf R$ with $\alpha\neq-\beta$.

Let
$$\frak d=\frak h\oplus\frak k\oplus\frak n.$$
Hence $\frak d=\frak k\oplus\frak g=\frak h\oplus\frak g'$.
Define a skew symmetric bilinear product $[\cdot,\cdot]$ on $\frak d$ by
\begin{equation}\label{Lie bracket}
\begin{array}{c}
\ [h_\lambda, h_\mu]=0, \ \ \ [h_\lambda,k_\mu]=0, \ \ \ [k_\lambda,k_\mu]=0,\\
\ [h_\lambda, x_\alpha]=-(\Phi_-\lambda|\alpha)x_\alpha,\ \ \  [k_\lambda,x_\alpha]=(\Phi_+\lambda|\alpha)x_\alpha,\\
\ [x_\alpha,x_\beta]=2^{-1}([x_\alpha,x_\beta]_{\frak g}+[x_\alpha,x_\beta]_{\frak g'})
\end{array}\end{equation}
for all $h_\lambda, h_\mu\in\frak h$,  $k_\lambda, k_\mu\in\frak k$ and  $x_\alpha,x_\beta\in\frak n$.

\begin{lem}
(1) For any  elements $x_\alpha, x_\beta\in\frak n$ such that $\alpha+\beta\neq0$,
\begin{equation}\label{A}
[x_\alpha,x_\beta]=[x_\alpha,x_\beta]_{\frak g}=[x_\alpha,x_\beta]_{\frak g'}.
 \end{equation}

(2)  For any  elements $x_\alpha, x_\beta, x_\gamma\in\frak n$ such that $\alpha+\beta+\gamma\neq0$,
 \begin{equation}\label{B}
 [[x_\alpha,x_\beta],x_\gamma]=[[x_\alpha,x_\beta]_{\frak g},x_\gamma]_{\frak g}
 =[[x_\alpha,x_\beta]_{\frak g'},x_\gamma]_{\frak g'}.
 \end{equation}

(3) For any  elements $x_\alpha, x_\beta, x_\gamma\in\frak n$,
 \begin{equation}\label{D}
 [[x_\alpha,x_\beta],x_\gamma]=2^{-1}([[x_\alpha,x_\beta]_{\frak g},x_\gamma]_{\frak g}
 +[[x_\alpha,x_\beta]_{\frak g'},x_\gamma]_{\frak g'}).
 \end{equation}
 \end{lem}

\begin{proof}
(1) It follows by (\ref{C}) and (\ref{Lie bracket}).

(2)
 If $\alpha+\beta\neq0$ then
 $[x_\alpha,x_\beta]=[x_\alpha,x_\beta]_{\frak g}=[x_\alpha,x_\beta]_{\frak g'}\in\frak n$  by (\ref{A})
 and thus
 $$[[x_\alpha,x_\beta],x_\gamma]=[[x_\alpha,x_\beta]_{\frak g},x_\gamma]_{\frak g}=[[x_\alpha,x_\beta]_{\frak g'},x_\gamma]_{\frak g'}$$
 by (\ref{A}) since $(\alpha+\beta)+\gamma\neq0$.

 Suppose  $\alpha+\beta=0$. We may assume that $\alpha$ is a positive root. By (\ref{X}), $[x_\alpha,x_{-\alpha}]_{\frak g}=ah_\alpha\in\frak h$ and $[x_\alpha,x_{-\alpha}]_{\frak g'}=ak_{\alpha}\in\frak k$, where
 $a=(x_\alpha|x_{-\alpha})$, and
 $$\begin{aligned}
 \ [[x_\alpha,x_{-\alpha}],x_\gamma]&=2^{-1}([[x_\alpha,x_{-\alpha}]_{\frak g}+[x_\alpha,x_{-\alpha}]_{\frak g'},x_\gamma]
 =2^{-1}a([h_\alpha, x_\gamma]+[k_\alpha,x_\gamma])\\
 &=2^{-1}a(-(\Phi_-\alpha|\gamma)+(\Phi_+\alpha|\gamma))x_\gamma=a(\alpha|\gamma)x_\gamma.
 \end{aligned}$$
 Since
 $$\begin{aligned}
 \ [[x_\alpha,x_{-\alpha}]_{\frak g},x_\gamma]_{\frak g}&=a[h_\alpha,x_\gamma]_{\frak g}=a(\alpha|\gamma)x_\gamma\\
 [[x_\alpha,x_{-\alpha}]_{\frak g'},x_\gamma]_{\frak g'}&=a[k_\alpha,x_\gamma]_{\frak g'}=a(\alpha|\gamma)x_\gamma,
 \end{aligned}$$
 we have the result $[[x_\alpha,x_{-\alpha}],x_\gamma]=[[x_\alpha,x_{-\alpha}]_{\frak g},x_\gamma]_{\frak g}
 =[[x_\alpha,x_{-\alpha}]_{\frak g'},x_\gamma]_{\frak g'}$.

(3) If $\alpha+\beta+\gamma\neq0$ then (\ref{D}) holds by (\ref{B}). Suppose that $\alpha+\beta+\gamma=0$. Then
$\alpha+\beta=-\gamma\neq0$.
Hence $[x_\alpha,x_\beta]=[x_\alpha,x_\beta]_{\frak g}=[x_\alpha,x_\beta]_{\frak g'}\in\frak n$ by (\ref{A}).
It follows that
$$\begin{aligned}
\ [[x_\alpha,x_{\beta}],x_\gamma]&=
2^{-1}([[x_\alpha,x_{\beta}],x_\gamma]_{\frak g}+[[x_\alpha,x_{\beta}],x_\gamma]_{\frak g'})\\
&=2^{-1}([[x_\alpha,x_{\beta}]_{\frak g},x_\gamma]_{\frak g}+[[x_\alpha,x_{\beta}]_{\frak g'},x_\gamma]_{\frak g'}).
\end{aligned}$$
\end{proof}

\begin{thm}
The vector space $\frak d=\frak h\oplus\frak k\oplus\frak n$ is a Lie algebra with Lie bracket (\ref{Lie bracket}).
\end{thm}

\begin{proof}
It is enough to show that (\ref{Lie bracket}) satisfies the Jacobi identity, $$[[a,b],c]+[[b,c],a]+[[c,a],b]=0$$ for all
nonzero elements $a,b,c\in\frak d.$
If $a,b,c\in\frak h\oplus \frak k$ then $[[a,b],c]+[[b,c],a]+[[c,a],b]=0$ trivially. If one of $a,b,c$ is an element of $\frak n$,
say $a=h_\lambda, b=k_\mu$ and $c=x_\alpha$, then
$$\begin{aligned}
\ [[a,b],c]+[[b,c],a]+[[c,a],b]&=[[h_\lambda, k_\mu],x_\alpha]+[[k_\mu, x_\alpha], h_\lambda]+[[x_\alpha, h_\lambda],k_\mu]\\
&=(\Phi_+\mu|\alpha)(\Phi_-\lambda|\alpha)x_\alpha-(\Phi_-\lambda|\alpha)(\Phi_+\mu|\alpha)x_\alpha\\
&=0.
\end{aligned}$$
If two of $a,b,c$ are elements of $\frak n$,
say $a=h_\lambda, b=x_\alpha$ and $c=x_\beta$, then
$$\begin{aligned}
\ [[a,b],c]&+[[b,c],a]+[[c,a],b]=[[h_\lambda, x_\alpha],x_\beta]+[[x_\alpha,x_\beta], h_\lambda]+[[x_\beta, h_\lambda],x_\alpha]\\
&=\left\{\begin{aligned}
&-(\Phi_-\lambda|\alpha)[x_\alpha,x_\beta]+(\Phi_-\lambda|\alpha+\beta)[x_\alpha,x_\beta]+(\Phi_-\lambda|\beta)[x_\beta,x_\alpha], &\text{ if }\alpha+\beta\neq0,\\
&-(\Phi_-\lambda|\alpha)[x_\alpha,x_{-\alpha}]+(\Phi_-\lambda|-\alpha)[x_{-\alpha},x_\alpha], &\text{ if }\alpha+\beta=0,\end{aligned}\right.\\
&=0.
\end{aligned}$$
Finally, suppose that all  of $a,b,c$ are elements of $\frak n$,
say $a=x_\alpha$, $b=x_\beta$ and $c=x_\gamma$. Then, by (\ref{D}),
$$\begin{aligned}
\ [[a,b],c]&+[[b,c],a]+[[c,a],b]=[[x_\alpha,x_\beta],x_\gamma]+[[x_\beta,x_\gamma], x_\alpha]+[[x_\gamma, x_\alpha],x_\beta]\\
&=2^{-1}([[x_\alpha,x_\beta]_{\frak g},x_\gamma]_{\frak g}+[[x_\alpha,x_\beta]_{\frak g'},x_\gamma]_{\frak g'})\\
&\qquad\qquad+2^{-1}([[x_\beta,x_\gamma]_{\frak g}, x_\alpha]_{\frak g}+[[x_\beta,x_\gamma]_{\frak g'}, x_\alpha]_{\frak g'})\\
&\qquad\qquad\qquad\qquad+2^{-1}([[x_\gamma, x_\alpha]_{\frak g},x_\beta]_{\frak g}+[[x_\gamma, x_\alpha]_{\frak g'},x_\beta]_{\frak g'})\\
&=0
\end{aligned}$$
since $\frak g$ and $\frak g'$ are Lie algebras.
It completes the proof.
\end{proof}


\subsection{}
Let $\frak m$ be a Lie algebra and let  $s=\sum_ia_i\otimes b_i\in\frak m\otimes\frak m$. We give a notation:
$$[[s,s]]=[s_{12},s_{13}]+[s_{12},s_{23}]+[s_{13},s_{23}],$$
 where
 $$\begin{aligned}
 \ [s_{12},s_{13}]&=\sum_{i,j}[a_i,a_j]\otimes b_i\otimes b_j,\\
 [s_{12},s_{23}]&=\sum_{i,j}a_i\otimes[b_i,a_j]\otimes b_j,\\
 [s_{13},s_{23}]&=\sum_{i,j}a_i\otimes a_j\otimes[b_i,b_j].
 \end{aligned}$$

Let
\begin{equation}\label{R}
r=\sum_{\alpha\in\bf R^+}(x_{\alpha}\otimes x_{-\alpha}-x_{-\alpha}\otimes x_{\alpha})\in\frak n\otimes\frak n,
\end{equation}
where $(x_{\alpha}|x_{-\alpha})=1$. Note that $r$ is skew symmetric. It is well-known that
$[[r,r]]$ is a $\frak g$-invariant element of $\frak g\otimes\frak g\otimes\frak g$, that is, $x.[[r,r]]=0$ for all $x\in\frak g$, where the action of $\frak g$ on $\frak g\otimes\frak g\otimes\frak g$ is by the adjoint representation in each factor:
$$x_.(a\otimes b\otimes c)=[x,a]_\frak g\otimes b\otimes c+a\otimes[x,b]_\frak g\otimes c+a\otimes b\otimes [x,c]_\frak g.$$ (See \cite[Proposition 2.1.2, Example 2.3.7 and \S 2.1 B]{ChPr} and \cite[1.2]{HoLeT}.)
 Thus $[[r,r]]$ is also a $\frak g'$-invariant element of $\frak g'\otimes\frak g'\otimes\frak g'$ by
 (\ref{Gprime}). Here we prove the following statement.

\begin{prop}\label{MCYBE}
The element $[[r,r]]\in\frak d\otimes \frak d\otimes\frak d$ is also $\frak d$-invariant.
\end{prop}

\begin{proof}
It is easy to check that
$$x.[r_{12},r_{13}]=0,\  x.[r_{12},r_{23}]=0,\  x.[r_{13},r_{23}]=0$$ for each  $x\in\frak h\oplus\frak k$. For instance,
let $x=h_\lambda\in\frak h$ and note that
$$\begin{aligned}
\ [r_{12},r_{13}]
=\sum_{\alpha,\beta\in{\bf R}^+}[x_\alpha,x_\beta]&\otimes x_{-\alpha}\otimes x_{-\beta}
-[x_\alpha,x_{-\beta}]\otimes x_{-\alpha}\otimes x_{\beta}\\
&-[x_{-\alpha},x_\beta]\otimes x_\alpha\otimes x_{-\beta}
+[x_{-\alpha},x_{-\beta}]\otimes x_\alpha\otimes x_\beta.
\end{aligned}$$
If $Y$ is the first term of $[r_{12},r_{13}]$, then
$$\begin{aligned}
x.Y&=\sum h_\lambda.([x_\alpha,x_\beta]\otimes x_{-\alpha}\otimes x_{-\beta})\\
&=\sum ([h_\lambda, [x_\alpha,x_\beta]]\otimes x_{-\alpha}\otimes x_{-\beta}+
[x_\alpha,x_\beta]\otimes [h_\lambda,x_{-\alpha}]\otimes
x_{-\beta}+[x_\alpha,x_\beta]\otimes x_{-\alpha}\otimes [h_\lambda,x_{-\beta}])\\
&=\sum(-(\Phi_-\lambda|\alpha+\beta)+(\Phi_-\lambda|\alpha)+(\Phi_-\lambda|\beta))[x_\alpha,x_\beta]\otimes x_{-\alpha}\otimes x_{-\beta}\\
&=0.
\end{aligned}$$
It remains to prove that $x_\gamma.[[r,r]]=0$ for all  $x_\gamma\in\frak n$.
Observe that
$$\begin{aligned}
x_\gamma.[[r,r]]&=x_\gamma.([r_{12},r_{13}]+[r_{12},r_{23}]+[r_{13},r_{23}])\\
&=2^{-1}[x_\gamma.([r_{12},r_{13}]_{\frak g}+[r_{12},r_{23}]_{\frak g}+[r_{13},r_{23}]_{\frak g})+x_\gamma.([r_{12},r_{13}]_{\frak g'}+[r_{12},r_{23}]_{\frak g'}+[r_{13},r_{23}]_{\frak g'})]
\end{aligned}$$
by  (\ref{D}) and the third equation of (\ref{Lie bracket}), where the last two actions are the actions of $\frak g$ and $\frak g'$ respectively.
Hence $x_\gamma.[[r,r]]=0$ for all  $x_\gamma\in\frak n$  since $x_\gamma.[[r,r]]=0$ in $\frak g\otimes\frak g\otimes\frak g$ and $\frak g'\otimes\frak g'\otimes\frak g'$.
\end{proof}

\subsection{}\label{TORUCS}
Let ${\bf Q}$ and
${\bf P}$ denote the root  and the weight lattices respectively,
$G$  a connected  complex semisimple algebraic group associated with  $\frak g$, $H\subset G$ a torus associated with $\frak h$ and let ${\bf L}$ be the group of characters of $H$, as in \cite[1.1]{HoLeT}. Note that ${\bf L}$ is a lattice such that ${\bf Q}\subseteq{\bf L}\subseteq {\bf P}$ by \cite[1.2]{HoLeT} and that ${\bf L}$ spans $\frak h^*$ since the set of simple roots $\{\alpha_i\}_{i=1}^n$ forms a basis of $\frak h^*$.

Let $\mathcal{C}(\frak g)$ be the category of $\frak g$-module homomorphisms as morphisms and the following (left) $\frak g$-modules as objects: the finite dimensional
 $\frak g$-modules consisting of finite direct sums of finite dimensional irreducible highest
weight $\frak g$-modules  $V(\Lambda)$ with highest weight $\Lambda\in {\bf L}^+={\bf L}\cap {\bf P}^+$, where ${\bf P}^+=\{\lambda\in{\bf P}|\lambda(h_i)\in\Bbb Z_{\geq0}\text{ for all } i=1,\ldots,n\}$  the set of dominant integral weights.
Here by $M\in\mathcal{C}(\frak g)$ we mean that $M$ is an  object of $\mathcal{C}(\frak g)$.
Note that $\mathcal{C}(\frak g)$ is closed under  finite direct sums, finite tensor products and  the formation duals.
(If $M,N\in\mathcal{C}(\frak g)$ then $M\otimes N$ and $M^*$ have the left $\frak g$-module structure with
$$\begin{aligned}
x(y\otimes z)&=(xy)\otimes z+y\otimes(xz),\\
(xf)(y)&=-f(xy)
\end{aligned}$$
 for all $f\in M^*$, $x\in\frak g$, $y\in M$ and $z\in N$.)
Moreover every  $M\in\mathcal{C}(\frak g)$ can be written by a direct sum of weight spaces $M=\oplus_{\mu\in{\bf L}}M_\mu$, where
\begin{equation}\label{H}
M_\mu=\{v\in M\ |\ h_\lambda v=(\lambda|\mu)v \text{ for }h_\lambda\in \frak h\}.
\end{equation}

Given $M=\oplus_{\mu\in{\bf L}}M_\mu\in\mathcal{C}(\frak g)$, note that each element  $f\in(M_\mu)^*$ is considered as the element
$f'\in M^*$ defined by
$$f'|_{M_\mu}=f,\ \ f'(M_\nu)=0\ \text{for all }\mu\neq\nu\in{\bf L}.$$
For instance, for $0\neq x\in M_\mu$, the dual map $x^*$ defined in $M_\mu$ is identified with the dual map $x^*$ defined in $M$.

\begin{lem}\label{WEIG}
Let $M=\oplus_{\nu\in{\bf L}}M_\nu\in\mathcal{C}(\frak g)$.

(1) Let $f\in (M_\mu)^*$. Then $f\in (M^*)_{-\mu}$.
In particular, if $f\in (M^*)_\mu$, $x\in M_\nu$ and $f(x)\neq0$ then $-\mu=\nu$.

(2) There exists a right $\frak g$-module structure in $M^*$: For $f\in M^*$, $x\in\frak g$ and $y\in M$,
$$(fx)(y)=f(xy).$$
In particular, if $f\in (M_\mu)^*$ and $x_\alpha\in\frak n$ then $fx_\alpha\in (M^*)_{-\mu+\alpha}$.
\end{lem}

\begin{proof}
(1) For any $h_\lambda\in\frak h$ and $y\in M_\mu$,
$$(h_\lambda f)(y)=-f(h_\lambda y)=-(\lambda|\mu)f(y)=(\lambda|-\mu)f(y)
$$
and thus $h_\lambda f=(\lambda|-\mu)f$. It follows that $f\in (M^*)_{-\mu}$.

Let $\{y_{\nu i}\}_i$ be a basis of $M_\nu$ for each $\nu\in{\bf L}$. Thus the set $\{y_{\nu i}\}_{\nu,i}$ is a basis of $M=\oplus_{\nu\in{\bf L}}M_\nu$
and $f\in (M^*)_\mu$ is uniquely expressed by a $\Bbb C$-linear combination $f=\sum_{\nu,i}a_{\nu i}y^*_{\nu i}$ for some
$a_{\nu i}\in\Bbb C.$ Since $f\in (M^*)_\mu$ and the dual map $y^*_{\nu i}$ is an element of $(M^*)_{-\nu}$ by the previous result,
we have that $a_{\nu i}=0$ for all $\nu i$ with $-\nu\neq \mu$ and thus $f=\sum_i a_{-\mu i}y^*_{-\mu i}$.
It follows that $f(M_\nu)=0$ for all $-\mu\neq \nu$.
The final statement is now clear:
since $f(x)\neq0$, $x\in M_{-\mu}$ and thus $\nu=-\mu$.

(2) For any $f\in M^*$, $x, x'\in\frak g$ and $y\in M$,
 $$\left((fx)x'-(fx')x\right)(y)=f([x,x']_\frak g y)=(f[x,x']_\frak g)(y).$$
 Hence $M^*$ is a right $\frak g$-module.
 If $f\in (M_\mu)^*$, $x_\alpha\in\frak n$, $h_\lambda\in\frak h$ and $y\in M_{-\alpha+\mu}$, then
 $$
 \left(h_\lambda(fx_\alpha)\right)(y)=-f(x_\alpha h_\lambda y)=-(\lambda|-\alpha+\mu)(fx_\alpha)(y).
 $$
 Hence $fx_\alpha\in (M^*)_{-\mu+\alpha}$.
\end{proof}

The following proposition shows that $M\in\mathcal{C}(\frak g)$ has a (left) $\frak d$-module structure.

\begin{prop}
Let $M=\oplus_{\mu\in{\bf L}}M_\mu$ be a finite direct sum of finite dimensional irreducible highest
weight $\frak g$-modules  $V(\Lambda)$ with highest weight $\Lambda\in {\bf L}^+$. Define an action of $\frak d$ on $M$ as follows:
For $z_\nu\in M_\nu$ and $h_\lambda\in\frak h$, $k_\lambda\in\frak k$, $x_\alpha\in\frak n$,
\begin{equation}\label{E}
\begin{array}{l}
h_\lambda\cdot z_\nu={(\Phi_+\nu| \lambda)}z_\nu, \\
 k_\lambda\cdot z_\nu={-(\Phi_-\nu|\lambda)}z_\nu,\\
x_\alpha\cdot z_\nu=x_\alpha z_\nu,
\end{array}
\end{equation}
where  the right hand side $x_\alpha z_\nu$ of the third equation is the action inside $\frak g$.
Then $M$ is a $\frak d$-module with action (\ref{E}).
\end{prop}

\begin{proof}
By (\ref{Lie bracket}) and (\ref{E}), it is easy to check that
$$\begin{aligned}
\ [h_\lambda,h_\mu]\cdot z_\nu&=0=h_\lambda\cdot (h_\mu\cdot z_\nu)-h_\mu\cdot (h_\lambda\cdot z_\nu), &&\\
\ [h_\lambda,k_\mu]\cdot z_\nu&=0=h_\lambda\cdot( k_\mu\cdot z_\nu)-k_\mu\cdot( h_\lambda\cdot z_\nu),&&\\
\ [k_\lambda,k_\mu]\cdot z_\nu&=0=k_\lambda\cdot( k_\mu\cdot z_\nu)-k_\mu\cdot( k_\lambda\cdot z_\nu). &&\\
\end{aligned} $$
Now observe that
$$\begin{aligned}
\  [h_\lambda, x_\alpha]\cdot z_\nu&={-(\Phi_-\lambda|\alpha)}x_\alpha\cdot z_\nu&&(\text{by }(\ref{Lie bracket}))\\
&=(\Phi_+\alpha|\lambda)x_\alpha z_\nu&&(\text{by (\ref{Phi}) and (\ref{E})})\\
&= h_\lambda\cdot( x_\alpha z_\nu) -x_\alpha( h_\lambda\cdot z_\nu),&&(\text{by (\ref{E})})\\
&= h_\lambda\cdot( x_\alpha\cdot z_\nu) -x_\alpha\cdot( h_\lambda\cdot z_\nu).&&(\text{by (\ref{E})})
\end{aligned} $$
Similarly, we have
$$\ [k_\lambda, x_\alpha]\cdot z_\nu={(\Phi_+\lambda|\alpha)}x_\alpha z_\nu=-(\Phi_-\alpha|\lambda)x_\alpha z_\nu= k_\lambda \cdot (x_\alpha\cdot z_\nu) -x_\alpha\cdot( k_\lambda\cdot z_\nu).$$

Let  $x_\alpha, x_\beta\in\frak n$.
If $\alpha+\beta\neq0$ then $[x_\alpha,x_\beta]\in\frak n$ and
$$[x_\alpha,x_\beta]\cdot z_\nu=[x_\alpha,x_\beta]_{\frak g}z_\nu=x_\alpha\cdot( x_\beta\cdot z_\nu)-x_\beta\cdot( x_\alpha\cdot z_\nu)$$
by (\ref{A}). Suppose that $\alpha+\beta=0$. We may assume that $\alpha\in{\bf R}^+$. Then $[x_\alpha,x_{-\alpha}]_{\frak g}=(x_\alpha|x_{-\alpha})h_\alpha\in\frak h$ and $[x_\alpha,x_{-\alpha}]_{\frak g'}=(x_{\alpha}|x_{-\alpha})k_\alpha\in\frak k$
by (\ref{X}) and
 $$\begin{aligned}
\ [x_\alpha,x_{-\alpha}]\cdot z_\nu&=2^{-1}(x_{\alpha}|x_{-\alpha})(h_\alpha+k_\alpha)\cdot z_\nu\\
& =(x_{\alpha}|x_{-\alpha})(\alpha|\nu)z_\nu=[x_\alpha,x_{-\alpha}]_{\frak g}z_\nu \\
&=x_\alpha\cdot( x_{-\alpha}\cdot z_\nu)- x_{-\alpha}\cdot( x_\alpha\cdot z_\nu).
 \end{aligned}$$
It completes the proof.
\end{proof}

\begin{notn}
(1) Denote by $\mathcal{C}(\frak d)$  the category $\mathcal{C}(\frak g)$ whenever all objects   $M\in\mathcal{C}(\frak g)$ are considered as the $\frak d$-modules $M$ with action (\ref{E}). For simplicity, by $M\in\mathcal{C}(\frak d)$ we mean that $M$ is an object of $\mathcal{C}(\frak d)$ and we will omit the dot $\lq\cdot$\rq \ denoting the module action of $\frak d$ on
$M\in\mathcal{C}(\frak d)$.

(2)
Let $U(\frak d)$ be the universal enveloping algebra of $\frak d$. Note that a vector space $M$ is a $\frak d$-module if and only if $M$ is a $U(\frak d)$-module.
For $M\in\mathcal{C}(\frak d), v\in M$ and $f\in M^*$,
define
$$c_{f,v}=c_{f,v}^M:U(\frak d)\longrightarrow \Bbb C,\ \ \ c_{f,v}(z)=f(zv)$$
and let $\Bbb C[G]$ be the $\Bbb C$-vector space spanned by all elements of the form
$$c_{f,v}^M,\ \ \ M\in\mathcal{C}(\frak d),\  v\in M, \ f\in M^*.$$
\end{notn}

Recall the definitions of a bigraded Hopf algebra and a (commutative) Poisson Hopf algebra given in \cite[2.1]{HoLeT} and \cite[Definition 1.1]{ChOh1}.
Let $K$ be an (additive) abelian group. A Hopf algebra $A=(A,\iota,m,\epsilon,\Delta, S)$ over a field ${\bf k}$ is
said to be
a {\it $K$-bigraded Hopf algebra} if it is equipped with a  $K\times K$-grading
$A=\underset{(\lambda,\mu)\in{K}\times{K}}{\bigoplus}A_{\lambda,\mu}$  such that
\begin{itemize}
\item[(i)] ${\bf k}\subseteq A_{0,0},\ \ \ A_{\lambda,\mu}A_{\lambda',\mu'}\subseteq A_{\lambda+\lambda', \mu+\mu'}$.
\item[(ii)] $\Delta(A_{\lambda,\mu})\subseteq \sum_{\nu\in{K}}A_{\lambda,\nu}\otimes A_{-\nu,\mu}$.
\item[(iii)] $\lambda\neq-\mu$ implies $\epsilon(A_{\lambda,\mu})=0$.
\item[(iv)] $S(A_{\lambda,\mu})\subseteq A_{\mu,\lambda}$.
\end{itemize}
A Hopf algebra $A=(A,\iota,m,\epsilon,\Delta, S)$ is said to be a {\it Poisson Hopf algebra} if there exists a skew symmetric bilinear product
$\{\cdot,\cdot\}$ on $A$ if $(A,\{\cdot,\cdot\})$ is a Poisson algebra such that
\begin{itemize}
\item[(v)]  $\Delta$ is a Poisson algebra homomorphism. Note that the Poisson bracket in $A\otimes A$ is given as follows: $$\{a_1\otimes a_2,b_1\otimes b_2\}=\{a_1,b_1\}\otimes a_2b_2+a_1b_1\otimes\{a_2,b_2\}.$$
\end{itemize}
A Poisson Hopf algebra $A=(A,\iota,m,\epsilon,\Delta, S, \{\cdot,\cdot\})$ is said to be {\it $K$-bigraded} if $A$ satisfies (i)-(v) and
\begin{itemize}
\item[(vi)] $ \{A_{\lambda,\mu},A_{\lambda',\mu'}\}\subseteq A_{\lambda+\lambda', \mu+\mu'}$.
\end{itemize}

\begin{thm}\label{ALGEB}
The vector space  $\Bbb C[G]$ is an ${\bf L}$-bigraded commutative Hopf algebra with the following structure:
For $M, N\in\mathcal{C}(\frak d)$,
\begin{equation}\label{J}
\begin{aligned}
c_{f,v}^M+c_{g,w}^N&=c_{(f, g), (v, w)}^{M\oplus N}, & c_{f,v}^Mc_{g,w}^N&=c_{f\otimes g, v\otimes w}^{M\otimes N}, \\
\epsilon(c_{f,v}^M)&=f(v), & S(c_{f,v}^M)&=c_{v,f}^{M^*},\\
\Delta(c_{f,v}^M)&=\sum_i c_{f,v_i}^M\otimes c_{g_i,v}^M,&&
\end{aligned}
\end{equation}
where $\{v_i\}$ and $ \{g_i\}$ are  bases for $M$ and $M^*$ such that $g_i(v_j)=\delta_{ij}$ for all $i,j$.

For $M\in\mathcal{C}(\frak d)$, set
$$C(M)=\Bbb C\langle c_{f,v}\ |\ f\in M^*,v\in M\rangle,\ \ \
C(M)_{\lambda,\mu}=\Bbb C\langle c_{f,v}\ |\ f\in (M^*)_\lambda,v\in M_\mu\rangle.$$
Then, for $\lambda,\mu\in{\bf L}$,
\begin{equation}\label{K}
\Bbb C[G]_{\lambda,\mu}=\sum_{M\in\mathcal{C}(\frak d)} C(M)_{\lambda,\mu}, \ \ \ \Bbb C[G]=\underset{{(\lambda,\mu)\in{\bf L}\times{\bf L}}}{\bigoplus}\Bbb C[G]_{\lambda,\mu}.
\end{equation}

Moreover $\Bbb C[G]$ is finitely generated as a $\Bbb C$-algebra.
\end{thm}

\begin{proof}
Note that $c^M_{f,v}$ is an element of the  Hopf dual $U(\frak d)^\circ$ since $M\in\mathcal{C}(\frak d)$ is finite dimensional. (Refer to \cite[I.9.5]{BrGo} for Hopf dual.) Since $\mathcal{C}(\frak d)$ is closed under finite direct
sums, finite tensor products and the formation  duals, $\Bbb C[G]$ is a sub-Hopf algebra of $U(\frak d)^\circ$ with Hopf structure
(\ref{J})  by \cite[I.7.3 and I.7.4]{BrGo}. Since $U(\frak d)$ is a co-commutative Hopf algebra, $U(\frak d)^\circ$ is a commutative Hopf algebra and thus $\Bbb C[G]$ is a commutative algebra.
By (\ref{J}) and (\ref{K}),  $\Bbb C[G]$ is ${\bf L}$-bigraded.

Let $\{\omega_1,\ldots,\omega_n\}$ be a basis of ${\bf L}$ such that ${\bf L}^+=\sum_i\Bbb Z_{\geq0}\omega_i$. For $\Lambda\in{\bf L}^+$, if $\Lambda=\sum_im_i\omega_i$ then $V(\Lambda)$ is isomorphic to a submodule of $\bigotimes_{i=1}^n\underbrace{V(\omega_i)\otimes\ldots\otimes V(\omega_i)}_{m_i}$ and thus $\Bbb C[G]$ is generated by all elements $c_{f,v}^{V(\omega_i)}$, $1\leq i\leq n$. Hence $\Bbb C[G]$
is finitely generated since each $V(\omega_i)$ is finite dimensional.
\end{proof}


\section{Poisson bracket on $\Bbb C[G]$}

\subsection{Hopf dual of co-Poisson Hopf algebra}
Refer to \cite[Definition 6.2.2]{ChPr} and \cite[Definition 2.2.9]{KoSo} for the definition of co-Poisson Hopf algebra. In \cite[Proposition 3.1.5]{KoSo}, there is a statement  that the Hopf dual of a co-Poisson Hopf algebra  is a Poisson Hopf algebra. Here we give a complete proof for this statement.

\begin{lem}\label{FA}
Let $I_1,\ldots, I_r$ be co-finite dimensional subspaces of a vector space $V$. Then $\cap_{i=1}^r I_i$ is also  co-finite.
\end{lem}

\begin{proof}
The kernel of the linear map
$$\psi:V\longrightarrow V/I_1\times\ldots\times V/I_r, \ \ \psi(x)=(x+I_1,\ldots, x+I_r)$$
is $\cap_{i=1}^r I_i$. Hence $\cap_{i=1}^r I_i$ is co-finite since $V/I_1\times\ldots\times V/I_r$ is finite dimensional.
\end{proof}

\begin{thm} \cite[Proposition 3.1.5]{KoSo}\label{L}
Let $A=(A,\iota,m,\epsilon,\Delta, S)$ be a co-Poisson co-commutative Hopf algebra with co-bracket $\delta$. Then its Hopf dual
$A^\circ$ is a Poisson Hopf algebra with Poisson bracket
\begin{equation}\label{M}
\{f,g\}=(f\otimes g)\delta, \ \ \ \ (f,g\in A^\circ).
\end{equation}
\end{thm}

\begin{proof} Note that $A^\circ$ is commutative since $A$ is co-commutative.
Let $I$ and $J$ be co-finite dimensional ideals of $A$ such that $f(I)=0$ and $g(J)=0$.
Then $K=I\cap J$ is co-finite by Lemma~\ref{FA} and thus $K\otimes A+A\otimes K$ is a co-finite ideal of $A\otimes A$ since
$A\otimes A/(K\otimes A+A\otimes K)$ is isomorphic to $(A/K)\otimes(A/K)$. Let $L$ be the ideal of $A$ generated by
$\delta^{-1}(K\otimes A+A\otimes K)\cap\Delta^{-1}(K\otimes A+A\otimes K)$. Then
$L$ is a co-finite ideal of $A$ by Lemma~\ref{FA} and $\delta(L)\subseteq K\otimes A+A\otimes K$
since $\delta(xy)=\delta(x)\Delta(y)+\Delta(x)\delta(y)$ for all $x,y\in A$. Hence
$(f\otimes g)\delta(L)=0$. It follows that $\{f,g\}\in A^\circ$.

By the co-Jacobi identity and the co-Leibniz identity, $A^\circ$ is a Poisson algebra.
For $x,y\in A$,
$$\begin{aligned}
m^*(\{f,g\})(x\otimes y)&=(f\otimes g)\delta(xy)\\
&=(f\otimes g)(\delta(x)\Delta(y)+\Delta(x)\delta(y))\\
&=\{m^*(f),m^*(g)\}(x\otimes y).
\end{aligned}$$
Hence $A^\circ$ is a Poisson Hopf algebra.
\end{proof}

Refer to \cite[Definition 1.3.1]{ChPr} for the definition of Lie bialgebra. Let $\frak m$ be a Lie algebra and let $U(\frak m)$
be the universal enveloping algebra of $\frak m$. It is well-known that $U(\frak m)$ has a Hopf structure $(U(\frak m),\iota,m, \epsilon,\Delta,S)$, where
$$\epsilon(x)=0,\ \ \Delta(x)=x\otimes1+1\otimes x, \ \ S(x)=-x$$ for all $x\in\frak m$.

\begin{cor}\label{O}
Let $\frak m$ be a Lie bialgebra with co-commutator $\delta'$ and let $U(\frak m)$ be the universal enveloping algebra of $\frak m$.

(1) In $U(\frak m)=(U(\frak m),\iota,m,\epsilon,\Delta, S)$, there exists a unique co-Poisson co-commutative Hopf algebra structure with co-bracket $\delta$ such that $\delta|_{\frak m}=\delta'$.

(2) The Hopf dual $U(\frak m)^\circ$ is a Poisson Hopf algebra with Poisson bracket (\ref{M}).
\end{cor}

\begin{proof}
(1) \cite[Proposition 6.2.3]{ChPr}. (The co-bracket $\delta$ is  a unique extension
$\delta:U(\frak m)\longrightarrow U(\frak m)\otimes U(\frak m)$ of $\delta'$ such that
$\delta(xy)=\Delta(x)\delta(y)+\delta(x)\Delta(y)$ for all $x,y\in U(\frak m)$.)

(2) It follows by (1) and Theorem~\ref{L}.
\end{proof}

\subsection{Poisson bracket of $U(\frak d)^\circ$}

Let
$$r=\sum_{\alpha\in\bf R^+}(x_{\alpha}\otimes x_{-\alpha}-x_{-\alpha}\otimes x_{\alpha})\in\frak n\otimes\frak n\ \ \ ((x_\alpha|x_{-\alpha})=1\text{ for each $\alpha\in{\bf R}^+$})$$
be the one given in (\ref{R}). Then $\frak d$ is a Lie bialgebra with Lie bracket (\ref{Lie bracket}) and co-commutator
\begin{equation}\label{P}
\delta'(x)=x.r
\end{equation}
for all $x\in \frak d$ by Proposition~\ref{MCYBE} and \cite[Proposition 2.1.2]{ChPr}, where the action $x.r$ is by the adjoint action of $x$ on each factor of $r$.

\begin{lem}
The co-commutator $\delta'$ given in (\ref{P}) is as follows:
\begin{equation}\label{PP}
\begin{array}{ll}
\delta'(h_\lambda)=0, &h_\lambda\in\frak h,\\
\delta'(k_\lambda)=0, &k_\lambda\in\frak k,\\
\delta'(x_\beta)= 2^{-1}x_\beta\wedge (h_\beta+k_\beta), &\beta\in\bf R^+,\\
\delta'(x_{-\beta})=2^{-1}x_{-\beta}\wedge (h_{\beta}+k_{\beta}), &\beta\in\bf R^+,
\end{array}\end{equation}
where $a\wedge b=a\otimes b-b\otimes a$.
\end{lem}

\begin{proof}
 For each $\alpha\in\bf R^+$,
$h_\lambda. (x_\alpha\wedge x_{-\alpha})=[h_\lambda,x_\alpha]\wedge x_{-\alpha}+x_\alpha\wedge[h_\lambda,x_{-\alpha}]=0$.
Similarly, $k_\lambda.(x_\alpha\wedge x_{-\alpha})=0$.
Hence $\delta'(h_\lambda)=0$ and $\delta'(k_\lambda)=0$.

 Let $\alpha,\beta,\alpha+\beta\in\bf R$. Thus $-(\alpha+\beta)\in\bf R$.
 Since  the root spaces of $\frak n$ are one dimensional
by \cite[Proposition 10.9]{ErWi},
$$[x_\beta,x_\alpha]=a_{\beta,\alpha}x_{\alpha+\beta}, \ \ \ [x_\beta, x_{-(\alpha+\beta)}]=b_{\beta,\alpha}x_{-\alpha}$$
for some $a_{\beta,\alpha}, b_{\beta,\alpha}\in\Bbb C$.
Then
 \begin{equation}\label{EQUL}
 a_{\beta,\alpha}=-b_{\beta,\alpha}
 \end{equation}
since
$$\begin{aligned}
a_{\beta,\alpha}&=(a_{\beta,\alpha}x_{\alpha+\beta}|x_{-(\alpha+\beta)})=([x_\beta,x_\alpha]|x_{-(\alpha+\beta)})=([x_\beta,x_\alpha]_\frak g|x_{-(\alpha+\beta)})\\
&=-(x_\alpha|[x_\beta, x_{-(\alpha+\beta)}]_\frak g)=-(x_\alpha|[x_\beta, x_{-(\alpha+\beta)}])=-(x_\alpha|b_{\beta,\alpha}x_{-\alpha})=-b_{\beta,\alpha}.
\end{aligned}$$
Hence $[x_\beta,x_\alpha]\wedge x_{-\alpha}=-x_{\alpha+\beta}\wedge [x_\beta,x_{-(\alpha+\beta)}]$.
It follows that
$$\begin{aligned}
x_\beta.r&=\sum_{\alpha\in\bf R^+} [x_\beta,x_\alpha]\wedge x_{-\alpha}
+x_\alpha\wedge[x_\beta,x_{-\alpha}]\\
&=\left\{{\begin{aligned}&x_\beta\wedge [x_\beta,x_{-\beta}],&\beta\in\bf R^+,\\
&[x_\beta,x_{-\beta}]\wedge x_\beta, &\beta\in\bf R^-,\end{aligned}}\right.
\\
&=\left\{{\begin{aligned}&2^{-1}x_\beta\wedge (h_\beta+k_\beta),&\beta\in\bf R^+,\\
&2^{-1}x_\beta\wedge (h_{-\beta}+k_{-\beta}),&\beta\in\bf R^-\end{aligned}}\right.
\end{aligned}$$
by (\ref{X}) and (\ref{Lie bracket}).
\end{proof}

By Corollary~\ref{O}, there exists a unique extension $\delta$ of $\delta'$ given in (\ref{P}) and  the Hopf dual $U(\frak d)^\circ$ is a Poisson Hopf algebra with Poisson bracket
$$\{a,b\}_r=(a\otimes b)\delta$$
for all $a,b\in U(\frak d)^\circ$. Let us find $\{\cdot,\cdot\}_r$ in the sub-Hopf algebra $\Bbb C[G]$ of $U(\frak d)^\circ$.

\begin{lem}
Let $M,N\in {\mathcal{C}}(\frak d)$ and let
$f\in (M^*)_\beta, g\in (N^*)_\gamma, p\in M_\eta, v\in N_\rho$ for $\beta,\gamma,\eta,\rho\in{\bf L}$.
Then
\begin{equation}\label{PPP}
\begin{array}{ll}
\{c_{f,p},c_{g,v}\}_r(h_\lambda)=0, &h_\lambda\in\frak h,\\
\{c_{f,p},c_{g,v}\}_r(k_\lambda)=0, &k_\lambda\in\frak k,\\
\{c_{f,p},c_{g,v}\}_r(x_\alpha)=(\alpha|\rho)f(x_\alpha p)g(v)-(\alpha|\eta)f(p)g(x_\alpha v), &\alpha\in\bf R.
\end{array}
\end{equation}
\end{lem}

\begin{proof}
By (\ref{PP}), $\{c_{f,p},c_{g,v}\}_r(h_\lambda)=0$, $\{c_{f,p},c_{g,v}\}_r(k_\lambda)=0$ and
$$\begin{aligned}
\{c_{f,p},c_{g,v}\}_r(x_\alpha)&=(c_{f,p}\otimes c_{g,v})\delta'(x_\alpha)\\
&=2^{-1}\left(f(x_\alpha p)g((h_\alpha+k_\alpha)v)-f((h_\alpha+k_\alpha)p)g(x_\alpha v)\right)\\
&=(\alpha|\rho)f(x_\alpha p)g(v)-(\alpha|\eta)f(p)g(x_\alpha v)
\end{aligned}$$
since $(h_\alpha+k_\alpha)v=2(\alpha|\rho)v$ and $(h_\alpha+k_\alpha)p=2(\alpha|\eta)p$ by (\ref{E}) and (\ref{Phi}).
\end{proof}

\begin{lem}\label{TRA}
For $M,N\in {\mathcal{C}}(\frak d)$, let
$\{f_i\}_{i=1}^r, \{u_i\}_{i=1}^r, \{g_j\}_{j=1}^s, \{v_j\}_{j=1}^s$ be $\Bbb C$-bases of $M^*, M, N^*, N$ respectively
such that $f_i(u_k)=\delta_{ik}$ and $g_j(v_\ell)=\delta_{j\ell}$. For any $\alpha\in\bf R^+$, suppose that
$$x_\alpha u_i=\sum_{k=1}^r a_{ik}u_k, \ \ x_{-\alpha}v_j=\sum_{\ell=1}^s b_{j\ell}v_\ell,$$
where $a_{ik},b_{j\ell}\in\Bbb C$.
Then
$$f_ix_\alpha =\sum_{k=1}^r a_{ki}f_k, \ \ g_jx_{-\alpha}=\sum_{\ell=1}^s b_{\ell j}g_\ell.$$
\end{lem}

\begin{proof}
Let $f_ix_\alpha =\sum_{k=1}^r d_kf_k$ for some $d_k\in\Bbb C$. Then
$$d_t=(\sum_{k=1}^r d_kf_k)(u_t)=(f_ix_\alpha)(u_t)=f_i(x_\alpha u_t)=f_i(\sum_{k=1}^r a_{tk}u_k)=a_{ti}.$$
Hence $f_ix_\alpha =\sum_{k=1}^r a_{ki}f_k$. Similarly, we have $g_jx_{-\alpha}=\sum_{\ell=1}^s b_{\ell j}g_\ell$.
\end{proof}

\begin{prop}\label{BPHA}
The function algebra $\Bbb C[G]$ is an ${\bf L}$-bigraded Poisson Hopf algebra with Poisson bracket
\begin{equation}\label{PBB}
\{c_{f,p}, c_{g,v}\}_r=[(\eta|\rho)-(\beta|\gamma)]c_{f,p} c_{g,v}
+2\sum_{\nu\in\bf R^+} (c_{f,x_\nu p}c_{g,x_{-\nu}v}-c_{fx_\nu,p}c_{gx_{-\nu},v})
\end{equation}
for weight vectors
$f\in (M^*)_\beta, g\in (N^*)_\gamma, p\in M_\eta, v\in N_\rho$ of $M,N\in{\mathcal{C}}(\frak d) $,
where $\beta,\gamma,\eta,\rho\in{\bf L}$.
\end{prop}

\begin{proof}
Let $\lambda_1,\ldots,\lambda_n$ be a $\Bbb C$-basis of $\frak h^*$. Then
$\frak B=\{h_{\lambda_i},k_{\lambda_i}, x_\alpha|i=1,\ldots,n,\alpha\in{\bf R}\}$ is a $\Bbb C$-basis of $\frak d$.
For convenience, we set $\frak B=\{y_1,\ldots, y_r\}$. Hence each $y_i$ is one of $h_\lambda,k_\lambda$ and $x_\alpha$, $0\neq\lambda\in\frak h^*,\alpha\in{\bf R}$.
 By the Poincare-Birkhoff-Witt theorem, $U(\frak d)$ has the $\Bbb C$-basis consisting of the standard monomials $Y=y_1^{s_1}\cdots y_r^{s_r}$ that is a product of $y_i\in\frak B$.  Let   $\ell(Y)$ be the number $s_1+\ldots+s_r$ of elements of $\frak B$ appearing in $Y$. We will call $\ell(Y)$ the length of $Y$. Denote by RH the right hand side of (\ref{PBB}). We proceed by induction on $\ell(Y)$ to show
$\{c_{f,p}, c_{g,v}\}_r(Y)=\text{RH}(Y)$.

If $\ell(Y)=0$ then $Y=1$, $\{c_{f,p}, c_{g,v}\}_r(1)=0$ and
$\text{RH}(1)=[(\eta|\rho)-(\beta|\gamma)]f(p)g(v)$. If $f(p)g(v)\neq0$ then $-\beta=\eta$ and $-\gamma=\rho$ by Lemma~\ref{WEIG}(1) and thus
$(\eta|\rho)-(\beta|\gamma)=0$. It follows that
\begin{equation}\label{ONE}
[(\eta|\rho)-(\beta|\gamma)]f(p)g(v)=0.
\end{equation}
 Hence $\{c_{f,p}, c_{g,v}\}_r(1)=0=\text{RH}(1)$.

Suppose that $\ell(Y)=1$. Then $Y=y_i$ for some $i$ and thus $Y=h_\lambda$, $Y=k_\lambda$ or $Y=x_\alpha$. If $Y=h_\lambda$ then
$$\begin{aligned}
\text{RH}(Y)&=[(\eta|\rho)-(\beta|\gamma)](f(h_\lambda p)g(v)+f(p)g(h_\lambda v))\\
&\qquad\qquad+2\sum(f(h_\lambda x_\nu p)g(x_{-\nu}v) +f(x_\nu p)g(h_\lambda x_{-\nu}v))\\
&\qquad\qquad-2\sum (f(x_\nu h_\lambda p)g(x_{-\nu}v)+f(x_\nu p)g(x_{-\nu}h_\lambda v))\\
&=[(\eta|\rho)-(\beta|\gamma)](f(h_\lambda p)g(v)+f(p)g(h_\lambda v))\\
&\qquad\qquad+2\sum f([h_\lambda, x_\nu]p)g(x_{-\nu}v)+f(x_\nu p)g([h_\lambda,x_{-\nu}]v)\\
&=[(\Phi_+\eta|\lambda)+(\Phi_+\rho|\lambda)][(\eta|\rho)-(\beta|\gamma)]f(p)g(v)\ \ \ \  (\text{by } (\ref{E}), \ref{Lie bracket}))\\
&=0.\ \ \ \ (\text{by }  (\ref{ONE}))
\end{aligned}
$$
 Hence $\{c_{f,p}, c_{g,v}\}_r(h_\lambda)=0=\text{RH}(h_\lambda)$ by (\ref{PPP}). Similarly
we have $\{c_{f,p}, c_{g,v}\}_r(k_\lambda)=\text{RH}(k_\lambda)$.
Let  $Y=x_\alpha$, $\alpha\in\bf R$. Suppose that $\alpha\in\bf R^+$. Then
$$\begin{aligned}
\text{RH}(x_\alpha)
&=[(\eta|\rho)-(\beta|\gamma)](f(x_\alpha p)g(v)+f(p)g(x_\alpha v))\\
&\qquad\qquad+2\sum f([x_\alpha, x_\nu]p)g(x_{-\nu}v)+f(x_\nu p)g([x_\alpha,x_{-\nu}]v)\\
&\qquad\qquad(\text{replace $[x_\alpha,x_\nu]$ by $a_{\alpha\nu}x_{\alpha+\nu}$, then  $[x_\alpha,x_{-(\alpha+\nu)}]=-a_{\alpha\nu}x_{-\nu}$ by } (\ref{EQUL}))\\
&=[(\eta|\rho)-(\beta|\gamma)](f(x_\alpha p)g(v)+f(p)g(x_\alpha v))+2f(x_\alpha p)g([x_\alpha,x_{-\alpha}]v)\\
&=[(\eta|\rho)-(\beta|\gamma)](f(x_\alpha p)g(v)+f(p)g(x_\alpha v))+f(x_\alpha p)g((h_\alpha+k_\alpha)v)\\
&=[(\eta|\rho)-(\beta|\gamma)+2(\alpha|\rho)]f(x_\alpha p)g(v)+[(\eta|\rho)-(\beta|\gamma)]f(p)g(x_\alpha v).
\end{aligned}
$$
If $f(x_\alpha p)g(v)\neq0$ then $-\beta=\alpha+\eta$ and  $-\gamma=\rho$ by Lemma~\ref{WEIG}(1),
 thus $f(p)g(x_\alpha v)=0$ and
$$\begin{aligned}
\text{RH}(x_\alpha)&=[(\eta|\rho)-(\beta|\gamma)+2(\alpha|\rho)]f(x_\alpha p)g(v)=(\alpha|\rho)f(x_\alpha p)g(v)\\
&=(\alpha|\rho)f(x_\alpha p)g(v) -(\alpha|\eta)f(p)g(x_\alpha v).
\end{aligned}$$
Similarly, if $f(p)g(x_\alpha v)\neq0$ then $-\beta=\eta$ and $-\gamma=\alpha+\rho$ by Lemma~\ref{WEIG}(1),  thus
$f(x_\alpha p)g(v)=0$ and
$$\begin{aligned}
\text{RH}(x_\alpha)&=[(\eta|\rho)-(\beta|\gamma)]f(p)g(x_\alpha v)=-(\alpha|\eta)f(p)g(x_\alpha v)\\
&=(\alpha|\rho)f(x_\alpha p)g(v)-(\alpha|\eta)f(p)g(x_\alpha v).
\end{aligned}$$
It follows that
$$\text{RH}(x_\alpha)=(\alpha|\rho)f(x_\alpha p)g(v)-(\alpha|\eta)f(p)g(x_\alpha v)$$
and thus $\{c_{f,p}, c_{g,v}\}_r(x_\alpha)=\text{RH}(x_\alpha)$ for $\alpha\in\bf R^+$ by (\ref{PPP}).
Similarly, we have $\{c_{f,p}, c_{g,v}\}_r(x_{-\alpha})=\text{RH}(x_{-\alpha})$ for $\alpha\in\bf R^+$.

Suppose that $\ell(Y)=s>1$ and that $\{c_{f,p}, c_{g,v}\}_r(Z)=\text{RH}(Z)$ for all monomials $Z$ with $\ell(Z)<s$.
Note that $Y=Xy$ for some monomial $X$ with $\ell(X)=s-1$ and $y\in\frak B$.
 Note that
\begin{equation}\label{PDEL}
\Delta(c_{f,p})=\sum_ic_{f,u_i}\otimes c_{f_i,p},\ \ \Delta(c_{g,v})=\sum_ic_{g,v_i}\otimes c_{g_i,v},
\end{equation}
where $\{f_i\}_{i=1}^a, \{u_i\}_{i=1}^a, \{g_i\}_{i=1}^b, \{v_i\}_{i=1}^b$ are $\Bbb C$-bases of $M^*,M,N^*,N$ such that
$f_i(u_j)=\delta_{ij}$ and $g_i(v_j)=\delta_{ij}$. We may assume that $f_i,u_i,g_i,v_i$ are all weight vectors and
denote by $\eta_i$ and $\rho_i$ the weights of $u_i$ and $v_i$ respectively. Hence the weights of $f_i$ and $g_i$ are $-\eta_i$
and $-\rho_i$ respectively by Lemma~\ref{WEIG}(1).
Then
$$\begin{aligned}
&\{c_{f,p},c_{g,v}\}_r(Xy)=\sum_{i,j} \{c_{f,u_i}, c_{g,v_j}\}_r(X)( c_{f_i,p}c_{g_j,v})(y)+(c_{f,u_i} c_{g,v_j})(X)\{c_{f_i,p},c_{g_j,v}\}_r(y)\\
&\qquad\qquad(\text{since } \Delta(\{c_{f,p},c_{g,v}\}_r)=\sum_{i,j} \{c_{f,u_i}, c_{g,v_j}\}_r\otimes c_{f_i,p}c_{g_j,v}+c_{f,u_i} c_{g,v_j}\otimes \{c_{f_i,p},c_{g_j,v}\}_r)\\
&=\sum_{i,j}[(\eta_i|\rho_j)-(\beta|\gamma)](c_{f,u_i} c_{g,v_j})(X)( c_{f_i,p}c_{g_j,v})(y)\\
&\qquad+2\sum_{i,j}\sum_{\alpha\in\bf R^+}(c_{f,x_\alpha u_i} c_{g,x_{-\alpha}v_j}-c_{fx_\alpha,u_i} c_{gx_{-\alpha},v_j})(X)( c_{f_i,p}c_{g_j,v})(y)\\
&+\sum_{i,j}[(\eta|\rho)-(\eta_i|\rho_j)](c_{f,u_i} c_{g,v_j})(X)( c_{f_i,p}c_{g_j,v})(y)\\
&\qquad+2\sum_{i,j}\sum_{\alpha\in\bf R^+}(c_{f,u_i} c_{g,v_j})(X)(c_{f_i,x_\alpha p} c_{g_j,x_{-\alpha}v}-c_{f_ix_\alpha,p} c_{g_jx_{-\alpha},v})(y)\\
&\qquad\qquad(\text{by induction hypothesis})\\
&=\sum_{i,j}[(\eta|\rho)-(\beta|\gamma)](c_{f,u_i} c_{g,v_j})(X)( c_{f_i,p}c_{g_j,v})(y)\\
&\qquad+2\sum_{\alpha\in\bf R^+}\sum_{i,j}(c_{f,u_i} c_{g,v_j})(X)(c_{f_i,x_\alpha p} c_{g_j,x_{-\alpha}v})(y)
-(c_{fx_\alpha,u_i} c_{gx_{-\alpha},v_j})(X)( c_{f_i,p}c_{g_j,v})(y)\\
&\qquad+2\sum_{\alpha\in\bf R^+}\sum_{i,j}(c_{f,x_\alpha u_i} c_{g,x_{-\alpha}v_j})(X)(c_{f_i, p} c_{g_j,v})(y)
-(c_{f,u_i} c_{g,v_j})(X)( c_{f_ix_\alpha,p}c_{g_jx_{-\alpha},v})(y)\\
&=\text{RH}(Xy)\\
&\qquad+2\sum_{\alpha\in\bf R^+}\sum_{i,j}(c_{f,x_\alpha u_i} c_{g,x_{-\alpha}v_j})(X)(c_{f_i, p} c_{g_j,v})(y)
-(c_{f,u_i} c_{g,v_j})(X)( c_{f_ix_\alpha,p}c_{g_jx_{-\alpha},v})(y)\\
&\qquad\qquad(\text{by } (\ref{PDEL}))\\
&=\text{RH}(Xy)\\
&\qquad+2\sum_{\alpha\in\bf R^+}\sum_{i,j}\sum_{k,\ell}a_{ik}b_{j\ell}(c_{f, u_k} c_{g,v_\ell})(X)(c_{f_i, p} c_{g_j,v})(y)
-a_{ki}b_{\ell j}(c_{f,u_i} c_{g,v_j})(X)( c_{f_k,p}c_{g_\ell,v})(y)\\
&\qquad\qquad(\text{replace  $x_\alpha u_i$ and $x_{-\alpha}v_j$ by $\sum_{k=1}^r a_{ik}u_k$ and $\sum_{\ell=1}^s b_{j\ell}v_\ell$ and use Lemma } \ref{TRA})\\
&=\text{RH}(Xy).
\end{aligned}$$
Hence we have (\ref{PBB}).

It is easily observed that the Poisson bracket of $\Bbb C[G]$ preserves the ${\bf L}$-bigrading by (\ref{PBB}) and Lemma~\ref{WEIG}.  It follows that $\Bbb C[G]$ is an $\bf L$-bigraded Poisson Hopf algebra by Theorem~\ref{ALGEB} and Theorem~\ref{L}.
\end{proof}

\begin{lem}\label{PADD}
Let $A_1=(A,\iota,m,\epsilon,\Delta, S, \{\cdot,\cdot\}_1)$ and $A_2=(A,\iota,m,\epsilon,\Delta, S, \{\cdot,\cdot\}_2)$ be  $K$-bigraded Poisson Hopf algebras over a field ${\bf k}$. Define a bilinear product $\{\cdot,\cdot\}$ on $A$ by
$$\{a,b\}=\{a,b\}_1+\{a,b\}_2$$
for $a,b\in A$. Then $A=(A,\iota,m,\epsilon,\Delta, S, \{\cdot,\cdot\})$ is also a $K$-bigraded Poisson Hopf algebra.
\end{lem}

\begin{proof}
It is easy to see that $A$ is a $K$-bigraded Poisson algebra with the Poisson bracket $\{\cdot,\cdot\}$.
Since $\Delta$ is a Poisson algebra homomorphism in $A_1$ and $A_2$, it is also a Poisson algebra homomorphism in $A$.
Hence $A=(A,\iota,m,\epsilon,\Delta, S, \{\cdot,\cdot\})$ is also a $K$-bigraded Poisson Hopf algebra.
\end{proof}

\begin{thm}\label{PMSTRU}
The function algebra $\Bbb C[G]$ is an ${\bf L}$-bigraded Poisson Hopf algebra with Poisson bracket
\begin{equation}\label{MPBB}
\{c_{f,w}, c_{g,v}\}=[(\Phi_+\eta|\rho)-(\Phi_+\beta|\gamma)]c_{f,w} c_{g,v}
+2\sum_{\nu\in\bf R^+} (c_{f,x_\nu w}c_{g,x_{-\nu}v}-c_{fx_\nu,w}c_{gx_{-\nu},v})
\end{equation}
for weight vectors
$f\in (M^*)_\beta, g\in (N^*)_\gamma, w\in M_\eta, v\in N_\rho$ of $M,N\in{\mathcal{C}}(\frak d) $,
\end{thm}

\begin{rem}\label{Phil}
(1) Suppose that the skew-symmetric bilinear form $u$ in \ref{SKEW} is equal to zero. Then we obtain (\ref{BPHA}) from (\ref{MPBB}) by (\ref{Phi}). Moreover the subspace $\langle h_\lambda-k_\lambda|\lambda\in\frak h^*\rangle$ of $\frak d$ is a Lie ideal of $\frak d$ and $\frak d/\langle h_\lambda-k_\lambda|\lambda\in\frak h^*\rangle\cong \frak g$ by (\ref{Lie bracket}).
Thus $\frak d$ and (\ref{MPBB}) are  generalizations of $\frak g$ and (\ref{BPHA}) respectively.

(2) One should compare (\ref{BPHA}) and (\ref{MPBB}) with \cite[(6) of I.8.16 Theorem]{BrGo} and \cite[Corollary 3.10]{HoLeT} respectively. Namely, the Poisson Hopf algebras $\Bbb C[G]$ with Poisson brackets (\ref{BPHA}) and (\ref{MPBB}) can be considered
as Poisson versions of the quantized algebras $\mathcal{O}_q(G)$ in \cite[I.8.16 Theorem]{BrGo} and $\Bbb C_{q,p}[G]$
in \cite{HoLeT} respectively.
\end{rem}

\begin{proof}[Proof of Theorem~\ref{PMSTRU}]
Note that $\Bbb C[G]$ is an ${\bf L}$-bigraded commutative Hopf algebra and that  $u$ in \ref{SKEW} is a skew symmetric bilinear form on $\bf L$. Define a bilinear product
$\{\cdot,\cdot\}_u$ on $\Bbb C[G]$ by
$$\{a_{\lambda,\mu},b_{\lambda',\mu'}\}_u=[u(\mu,\mu')-u(\lambda,\lambda')]a_{\lambda,\mu}b_{\lambda',\mu'}$$
for $a_{\lambda,\mu}\in\Bbb C[G]_{\lambda,\mu}$ and $b_{\lambda',\mu'}\in\Bbb C[G]_{\lambda',\mu'}$.
Then it is easy to check that $\Bbb C[G]$ is an ${\bf L}$-bigraded Poisson  Hopf algebra with Poisson bracket $\{\cdot,\cdot\}_u$
since $u$ is skew-symmetric. Hence the result follows by Proposition~\ref{BPHA} and Lemma~\ref{PADD} since
$$u(\eta,\rho)+(\eta|\rho)=(\Phi_+\eta|\rho),\ \ u(\beta,\gamma)+(\beta|\gamma)=(\Phi_+\beta|\gamma)$$
by (\ref{Phi}).
\end{proof}



 \section{Poisson prime  ideals of $\Bbb C[G]$}

 In this section we restate  some basic definitions and properties for  $\Bbb C[G]$ using Poisson terminologies, that appear in the quantized algebra $\Bbb C_{q,p}[G]$ of \cite[\S4]{HoLeT}.
 \bigskip

 Assume throughout the section that, for each $\Lambda\in{\bf L}^+$,  $V(\Lambda)$ denotes the highest weight $\frak g$-module with the highest weight $\Lambda$ and thus $V(\Lambda)\in\mathcal{C}(\frak d)$ satisfies the action (\ref{E}).
 Let $M\in\mathcal{C}(\frak d)$. Then $M=\oplus_{\mu\in{\bf L}}M_\mu$. We often write $v_\mu$ for $v\in M_\mu$.

\begin{thm}\label{MMMM}
Set
 $$\Bbb C[G]^+=\sum_{\Lambda\in {\bf L}^+}\sum_{f\in V(\Lambda)^*}\Bbb C c^{V(\Lambda)}_{f,v_\Lambda},\ \
 \Bbb C[G]^-=\sum_{\Lambda\in {\bf L}^+}\sum_{f\in V(\Lambda)^*}\Bbb C c^{V(\Lambda)}_{f,v_{w_0\Lambda}},
 $$
 where $w_0$ is the longest element of  the Weyl group $W$ of $\frak g$.
The multiplication map $\Bbb C[G]^+\otimes \Bbb C[G]^-\longrightarrow \Bbb C[G]$ is surjective.
\end{thm}

\begin{proof}
It is proved by mimicking the proof of  \cite[2.2.1 of  Chapter 3]{KoSo} or \cite[9.2.2 Proposition]{Jos}.
\end{proof}

Denote by $U(\frak b^+)$ and $U(\frak b^-)$ the subalgebras of $U(\frak d)$ generated by
$$
h_\lambda\in\frak h,\ \  x_\alpha, \ \ \ \alpha\in{\bf R}^+$$
and $$k_\lambda\in\frak k,\ \  x_{-\alpha}, \ \ \ \alpha\in{\bf R}^+$$
respectively.

Let $M\in\mathcal{C}(\frak d)$.  Given a subset $X\subseteq M$, write $X^\bot$ for the orthogonal space of $X$ in $M^*$, that is,
$$X^\bot=\{f\in M^* \ |\ f(X)=0\}.$$
Let $w_0$ be the longest element of  the Weyl group $W$.
For $\Lambda\in {\bf L}^+$ and $y\in W$, define the ideals $I^+_y$ and $I^-_y$ of the $\bf L$-bigraded Poisson Hopf algebra $\Bbb C[G]$ with Poisson bracket (\ref{MPBB}):
$$
I^+_y=\langle c^{V(\Lambda)}_{f,v_\Lambda}\ |\ f\in (U(\frak b^+)V(\Lambda)_{y\Lambda})^\bot \rangle,  \ \ \
I^-_y=\langle c^{V(\Lambda)}_{f,v_{w_0\Lambda}}\ |\ f\in (U(\frak b^-)V(\Lambda)_{yw_0\Lambda})^\bot\rangle.
$$
For $w=(w_+,w_-)\in W\times W$, define
$$I_w=I_{w_+}^+ +I_{w_-}^-.$$

\begin{lem}\label{CPRIB}
For any $w=(w_+,w_-)\in W\times W$, $I_w$ is a  homogeneous Poisson ideal of $\Bbb C[G]$.
\end{lem}

\begin{proof}
Since  $I_{w_+}^+$ and $I_{w_-}^-$ are generated by homogeneous elements, they are homogeneous ideals. Moreover  $I_{w_+}^+$ and $I_{w_-}^-$ are Poisson ideals by (\ref{MPBB}). Hence $I_w$ is a  homogeneous Poisson ideal.
\end{proof}

For any $w=(w_+,w_-)\in W\times W$, define the following homogeneous elements in the Poisson algebra $\Bbb C[G]/I_w$:
$$
c_{w\Lambda}=c^{V(\Lambda)}_{f_{-w_+\Lambda},v_\Lambda}+I_w, \ \ \
\tilde c_{w\Lambda}=c^{V(\Lambda)^*}_{v_{w_-\Lambda}, f_{-\Lambda}}+I_w , \ \ \ \Lambda\in {\bf L}^+.
$$
Note that $\tilde c_{w\Lambda}$ is of the form $c^{V(\Lambda)}_{f_{-w_-w_0\Lambda}, v_{w_0\Lambda}}+I_w$ since $w_0(\bf R^+)=\bf R^-$ by \cite[A.1.4]{Jos}.

An element $a$ of a Poisson algebra $A$ is said to be {\it Poisson normal} if $\{a,A\}\subseteq aA$. We should compare the following lemma with \cite[Lemma 4.2]{HoLeT}.

\begin{lem}\label{CPRIBB}
Let $\Lambda\in{\bf L}^+$.

(1) $\{c_{w\Lambda}, c_{f_{-\lambda}, v_\mu}+I_w\}
=[(\Phi_+\Lambda|\mu)-(\Phi_+w_+\Lambda|\lambda)](c_{f_{-\lambda}, v_\mu}+I_w)c_{w\Lambda}$.

(2) $\{\tilde c_{w\Lambda},c_{f_{-\lambda}, v_\mu}+I_w\}
=[(\Phi_-w_-\Lambda|\lambda)-(\Phi_-\Lambda|\mu)](c_{f_{-\lambda}, v_\mu}+I_w)\tilde c_{w\Lambda}$.

In particular, both $c_{w\Lambda}$ and $\tilde c_{w\Lambda}$ are
Poisson normal elements of  $\Bbb C[G]/I_w$.
\end{lem}

\begin{proof}
(1)  It is proved immediately by (\ref{MPBB}).

(2)  It is proved immediately by (\ref{MPBB}) and (\ref{Phi}).
\end{proof}

Observe that the sets
$$\begin{array}{c}\mathcal{E}_{w}^+=\{\alpha c_{w\Lambda}\ |\ \alpha\in\Bbb C^\times,\ \Lambda\in {\bf L}^+\}, \ \ \
\mathcal{E}_{w}^-=\{\alpha \tilde c_{w\Lambda}\ |\ \alpha\in\Bbb C^\times,\ \Lambda\in {\bf L}^+\}\\
\mathcal{E}_w=\mathcal{E}_{w}^+\mathcal{E}_{w}^-
\end{array}$$
are multiplicatively closed sets in $\Bbb C[G]/I_w$.
Denote the localization by
$$\Bbb C[G]_w=(\Bbb C[G]/I_w)_{\mathcal{E}_w}.$$

\begin{thm}\label{CPRIBBB}
For any Poisson prime ideal $P$ of $\Bbb C[G]$, there exists a unique $w=(w_+,w_-)\in W\times W$ such that
$I_w\subseteq P$ and $(P/I_w)\cap\mathcal{E}_w=\varnothing$.
\end{thm}

\begin{proof}
The proof is parallel to that of \cite[Theorem 4.4]{HoLeT}. We repeat it for completion.
For $\Lambda\in{\bf L}^+$,
define an order relation on the weight vectors of $V(\Lambda)^*$ by $f\leq f'$ if $f'\in fU(\frak b^+)$.
This induces a partial ordering on the set of one dimensional weight spaces.
For $\Lambda\in {\bf L}^+$, set
$$\mathcal{D}(\Lambda)=\{f\in (V(\Lambda)^*)_{-\mu}\ |\ c^{V(\Lambda)}_{f,v_\Lambda}\notin  P\}.$$
We claim that $\mathcal{D}(\Lambda)\neq\varnothing$ for all $\Lambda\in {\bf L}^+$.
Let $\{g_i\}$ and $\{v_i\}$ be dual bases for $V(\Lambda)^*$ and $V(\Lambda)$ that consist of weight vectors.
Since we have
$$1=\epsilon(c^{V(\Lambda)}_{f_{-\Lambda},v_\Lambda})
=\sum_{i\in I}S(c^{V(\Lambda)}_{f_{-\Lambda},v_i})c^{V(\Lambda)}_{g_i,v_\Lambda}$$
there exists an index $i$ such that $c^{V(\Lambda)}_{g_i,v_\Lambda}\notin P$ and thus $g_i\in\mathcal{D}(\Lambda)$ as claimed.

Suppose that $f,f'\in\mathcal{D}(\Lambda)$ are both maximal elements and set
 $f\in (V(\Lambda)^*)_{-\mu}$, $f'\in (V(\Lambda)^*)_{-\mu'}$.
 By (\ref{MPBB}) and the maximality of $f$ in $\mathcal{D}(\Lambda)$, we have
 $$\{c^{V(\Lambda)}_{f,v_\Lambda}, c^{V(\Lambda)}_{f',v_\Lambda}\}
 =[(\Phi_+\Lambda|\Lambda)-(\Phi_+\mu|\mu')]c^{V(\Lambda)}_{f,v_\Lambda}c^{V(\Lambda)}_{f',v_\Lambda}\ \ (\text{mod}\ P).$$
 Using the same argument for the maximality of $f'$, we obtain
 $$\{c^{V(\Lambda)}_{f',v_\Lambda},c^{V(\Lambda)}_{f,v_\Lambda}\}
 =[(\Phi_+\Lambda|\Lambda)-(\Phi_+\mu'|\mu)]c^{V(\Lambda)}_{f,v_\Lambda}c^{V(\Lambda)}_{f',v_\Lambda},\ \ (\text{mod}\ P).$$
Thus  we have $(\Lambda|\Lambda)=(\mu|\mu')$ by adding the above  equations since
$c^{V(\Lambda)}_{f,v_\Lambda}c^{V(\Lambda)}_{f',v_\Lambda}\notin P$ and $u$ is skew symmetric.
It follows that there exists $w_\Lambda\in W$ such that
$\mu=w_\Lambda\Lambda=\mu'$ by \cite[Proposition 11.4]{Kac}. That is, there exists a unique (up to scalar multiplication)
maximal element $g_\Lambda$  in
$\mathcal{D}(\Lambda)$ with weight $-w_\Lambda\Lambda$.

For $\Lambda$, $\Lambda'\in {\bf L}^+$, consider the maximal elements
$g_\Lambda\in\mathcal{D}(\Lambda)$ and $g_{\Lambda'}\in\mathcal{D}(\Lambda')$ with weights $-w_\Lambda\Lambda$
and $-w_{\Lambda'}\Lambda'$ respectively. Applying the argument above to a pair of such elements
$c^{V(\Lambda)}_{g_\Lambda,v_\Lambda}$ and $c^{V(\Lambda')}_{g_{\Lambda'},v_{\Lambda'}}$,
 we get
$(w_\Lambda\Lambda| w_{\Lambda'}\Lambda')=(\Lambda|\Lambda')$ for all $\Lambda,\Lambda'\in {\bf L}^+$.
It follows that $w_\Lambda=w_{\Lambda'}$ by \cite[Lemma 5.1.6 of Chapter 3]{KoSo} and thus there exists a unique $w_+\in W$ such that
$f_{-w_+\Lambda}$ is the maximal element in $\mathcal{D}(\Lambda)$ for all $\Lambda\in {\bf L}^+$.
For  each $\Lambda\in{\bf L}^+$ and weight vector $g\in V(\Lambda)^*$, if $c_{g,v_\Lambda}^{V(\Lambda)}\in I_{w_+}^+\setminus P$ then
$g\in\mathcal{D}(\Lambda)$ and thus $g\leq f_{-w_+\Lambda}$. But $g\in (U(\frak b^+)V(\Lambda)_{w_+\Lambda})^\bot$. Thus $g(U(\frak b^+)V(\Lambda)_{w_+\Lambda})=0$ and thus $f_{-w_+\Lambda}(V(\Lambda)_{w_+\Lambda})=0$,  that is a contradiction.
It follows that $I_{w_+}^+\subseteq P$. Using the same argument as above we deduce the existence of a unique $w_-\in W$ such that $I_{w_-}^-\subseteq P$. Now it is easy to check
 $(P/I_{w})\cap\mathcal{E}_{w}=\varnothing$ since $P$ is a prime ideal.
\end{proof}

For a Poisson algebra $A$, denote by $\Pspec A$ (respectively, $\Pprim A$) the set of all Poisson prime ideals
(respectively, Poisson primitive ideals) of $A$. By Theorem~\ref{CPRIBBB}, we have the following result.

\begin{cor}\label{PPPP}
For any $w=(w_+,w_-)\in W\times W$, set
$$\aligned
\Pspec_w \Bbb C[G]&=\{P\in \Pspec \Bbb C[G] \ |\ I_w\subseteq P\text{ and }(P/I_w)\cap\mathcal{E}_w=\varnothing\}\\
\Pprim_w \Bbb C[G]&=\Pprim \Bbb C[G] \cap\Pspec_w \Bbb C[G].
\endaligned$$
Then
$$\aligned
\Pspec \Bbb C[G]&=\bigsqcup_{w\in W\times W}\Pspec_w \Bbb C[G]\\
\Pprim \Bbb C[G] &=\bigsqcup_{w\in W\times W}\Pprim_w \Bbb C[G]
\endaligned$$
and, for each $w\in W\times W$,
$$\Pspec_w \Bbb C[G]\cong \Pspec\Bbb C[G]_w,\ \ \
\Pprim_w \Bbb C[G]\cong \Pprim\Bbb C[G]_w.$$
\end{cor}


\section{Poisson adjoint action}

Let $A$ be a Poisson algebra. Recall the definition of Poisson $A$-module $M$ given in   \cite[Definition 1]{Oh5}:
 A vector space $M$ is said to be a Poisson $A$-module if
\begin{itemize}
\item[(i)]  $M$ is a module over the commutative algebra $A$
with module structure $$M\times A\longrightarrow M, (z,a)\mapsto za,$$
\item[(ii)]  $M$ is a  module over the Lie algebra $(A,\{\cdot,\cdot\})$
with module structure $$A\times M\longrightarrow M, (a,z)\mapsto a*z$$
\end{itemize}
such that
\begin{equation}\label{ADJA}
z\{a,b\}=a*(zb)-(a*z)b
\end{equation}  and
\begin{equation}\label{ADJB}
(ab)*z=(b*z)a+(a*z)b
\end{equation}  for all $a,b,\in A$ and $z\in M$.

\begin{defn}
Let $M$ be a Poisson module over a Poisson Hopf algebra $A=(A,\iota,m,\epsilon, \Delta,S)$. Then
{\it the Poisson adjoint action} on $M$ is defined by
$$\text{ad}_a(z)=\sum_{(a)} (a_1* z)S(a_2),\ \ \   a\in A, z\in M,$$
where  $\Delta(a)=\sum_{(a)}a_1\otimes a_2$.
\end{defn}

There exists a {\it canonical Poisson adjoint action} on $A$ given by
$$\text{ad}_a(z)=\sum_{(a)} \{a_1, z\}S(a_2),\ \ \   a, z\in A$$
since $A$ is a Poisson $A$-module with Poisson module structure such that
$a*z=\{a,z\}$ and $za$ is the multiplication in $A$ for all $a,z\in A$. Moreover the canonical
Poisson adjoint action is a derivation on $A$, that is,
\begin{equation}\label{ADJC}
\text{ad}_a(zy)=\text{ad}_a(z)y +  \text{ad}_a(y)z,\ \ \   a, z, y\in A.
\end{equation}

\begin{lem}\label{ADJD}
Let $M$ be a Poisson module over a Poisson Hopf algebra $A=(A,\iota,m,\epsilon, \Delta,S)$.

(1) $\text{ad}_a(z)=-\sum_{(a)} [S(a_2)*z]a_1$ for $a\in A$ and $z\in M$.

(2) $\text{ad}_{ab}=\epsilon(a)\text{ad}_b + \epsilon(b)\text{ad}_a$ for all $a,b\in A$.

(3) $\text{ad}_{\{a,b\}}=(\text{ad}_a)(\text{ad}_b)-(\text{ad}_b)(\text{ad}_a)$ for all $a,b\in A$.

(4) Define
$$\begin{array}{ll}
M\times A\longrightarrow M, &(z,a)\mapsto z\cdot a=\epsilon(a)z\\
A\times M\longrightarrow M, &(a,z)\mapsto a*'z=\text{ad}_a(z).
\end{array}$$
Then $(M,\cdot,*')$ is a Poisson $A$-module.
\end{lem}

\begin{proof}
Let $a,b\in A$ and $z\in M$.

(1) It follows immediately by the fact that
$$0=\epsilon(a)1*z=\sum(a_1S(a_2))*z=\sum(S(a_2)*z)a_1+\sum(a_1*z)S(a_2)$$
by (\ref{ADJB}).

(2) It follows immediately by the fact that
$$\begin{aligned}
\text{ad}_{ab}(z)&=\sum((a_1b_1)*z)S(a_2b_2)&&\\
&=\sum [(b_1*z)a_1+(a_1*z)b_2]S(a_2)S(b_2)&&\text{(by (\ref{ADJB}))}\\
&=\epsilon(a)\text{ad}_b(z)+\epsilon(b)\text{ad}_a(z).&&
\end{aligned}$$

(3) We show that $S$ is a Poisson anti-homomorphism, namely
\begin{equation}\label{ADJE}
S(\{a,b\})=-\{S(a),S(b)\}.
\end{equation}
Since
$$\begin{aligned}
0&=\{\epsilon(a)1,b\}=\{\sum S(a_1)a_2,b\}=\sum \{S(a_1),b\}a_2+\sum S(a_1)\{a_2,b\}\\
0&=\{a,\epsilon(b)1\}=\{a,\sum S(b_1)b_2\}=\sum \{a,S(b_1)\}b_2+\sum S(b_1)\{a,b_2\}
\end{aligned}$$
we have
\begin{equation}\label{ADJF}
\sum \{S(a_1),b\}a_2=-\sum S(a_1)\{a_2,b\}, \ \ \ \sum\{a,S(b_1)\}b_2=-\sum S(b_1)\{a,b_2\}.
\end{equation}
Note that
\begin{equation}\label{ADJG}
\Delta(\{a,b\})=\sum a_1b_1\otimes\{a_2,b_2\}+\sum\{a_1,b_1\}\otimes a_2b_2
\end{equation}
since $A$ is a Poisson Hopf algebra.
Since $\epsilon(\{a,b\})=0$ for all $a,b\in A$ by \cite[1.2]{ChOh1},
$$0=\epsilon(\{a,b\})1=\sum S(\{a,b\}_1)\{a,b\}_2=\sum S(\{a_1,b_1\})a_2b_2+\sum S(a_1b_1)\{a_2,b_2\}$$
by (\ref{ADJG}),
thus
\begin{equation}\label{ADJH}
\sum S(\{a_1,b_1\})a_2b_2=-\sum S(a_1b_1)\{a_2,b_2\}=-\{S(a_1),S(b_1)\}a_2b_2
\end{equation} by (\ref{ADJF}).
Therefore
$$\begin{aligned}
S(\{a,b\})&=S(\{\sum a_1 \epsilon(a_2), \sum b_1\epsilon( b_2)\})=\sum S(\{a_1,b_1\})\epsilon(a_2)\epsilon(b_2)&&\\
&=\sum S(\{a_1,b_1\})a_2S(a_3)b_2S(b_3)\\
&=-\sum \{S(a_1),S(b_1)\}a_2S(a_3)b_2S(b_3)&&\text{(by (\ref{ADJH}))}\\
&=-\sum\{S(a_1)\epsilon(a_2), S(b_1)\epsilon(b_2)\}=-\{S(a),S(b)\},
\end{aligned}$$
as claimed.

Now (3) follows  by the fact that
$$\begin{aligned}
\text{ad}_{\{a,b\}}(z)&=\sum(\{a,b\}_1*z)S(\{a,b\}_2)\\
&=\sum[(a_1b_1)*z]S(\{a_2,b_2\})+\sum[\{a_1,b_1\}*z]S(a_2b_2)&&\text{(by (\ref{ADJG}))}\\
&=\sum[(a_1*z)b_1+(b_1*z)a_1]\{S(b_2),S(a_2)\}&&\text{(by (\ref{ADJB}), (\ref{ADJE}))}\\
&\qquad+\sum[a_1*(b_1*z)-b_1*(a_1*z)]S(a_2)S(b_2)\\
&=\sum\left(S(b_2)*[(a_1*z)S(a_2)]-[S(b_2)*(a_1*z)]S(a_2)\right)b_1&&\text{(by (\ref{ADJA}))}\\
&\qquad-\sum(S(a_2)*[(b_1*z)S(b_2)]-[S(a_2)*(b_1*z)]S(b_2))a_1&&\text{(by (\ref{ADJA}))}\\
&\qquad+\sum[a_1*(b_1*z)-b_1*(a_1*z)]S(a_2)S(b_2)\\
&=-\text{ad}_b\text{ad}_a(z)+\text{ad}_a\text{ad}_b(z)&&\text{(by (1))}\\
&\qquad+\sum[(a_1S(a_2))*(b_1*z)]S(b_2)-\sum[(b_1S(b_2))*(a_1*z)]S(a_2)&&\text{(by (\ref{ADJB}))}\\
&=(\text{ad}_a\text{ad}_b-\text{ad}_b\text{ad}_a)(z)
\end{aligned}$$
since $\sum a_1S(a_2)=\epsilon(a)1$ and $\sum b_1S(b_2)=\epsilon(b)1$.

(4) Clearly $(M,\cdot)$ is a module over the commutative algebra $A$ and $(M,*')$ is a module
over the Lie algebra $(A,\{\cdot,\cdot\})$ by (3). Hence it is enough to prove that
$$z\cdot\{a,b\}=a*'(z\cdot b)-(a*'z)\cdot b,\ \ \ (ab)*'z=(b*'z)\cdot a+(a*'z)\cdot b.$$
The first equation follows from the fact $\epsilon(\{a,b\})=0$  by \cite[1.2]{ChOh1}
and the second equation follows from (2).
\end{proof}

\begin{thm}\label{ADJK}
 Let $A$ be a Poisson Hopf algebra and let $M$ be a Poisson $A$-module. Set
$$\begin{array}{l}
\mathcal{Z}(M)=\{z\in M\ |\ a*z=0\text{ for all }a\in A\}\\
M^{\text{ad}}=\{z\in M\ |\ \text{ad}_a(z)=0\text{ for all }a\in A\}.
\end{array}$$
Then $\mathcal{Z}(M)=M^\text{ad}.$
\end{thm}

\begin{proof}
If  $z\in \mathcal{Z}(M)$ then   $\text{ad}_a(z)=\sum_{(a)} (a_1*z)S(a_2)=0$ for all $a\in A$. Thus
 $\mathcal{Z}(M)\subseteq M^{\text{ad}}$.
Conversely, if $z\in M^{\text{ad}}$ then
$$a*z=\sum_{(a)} (a_1*z)\epsilon(a_2)=\sum_{(a)} (a_1*z)S(a_2)a_3=\sum_{(a)} \text{ad}_{a_1}(z)a_2=0$$
for all $a\in A$.
Thus $M^{\text{ad}}\subseteq \mathcal{Z}(M)$.
\end{proof}

 Let $\frak m$ be a Lie algebra and let $\mathcal{S}(\frak m)$ be the symmetric algebra of $\frak m$. It is well-known that  $\mathcal{S}(\frak m)$ is a Poisson Hopf algebra with
 $$\Delta(z)=z\otimes 1+1\otimes z, \ \epsilon(z)=0, \  S(z)=-z,\  \{z,y\}=[z,y]$$
for all $z,y\in\frak m$. The canonical Poisson adjoint action on $\mathcal{S}(\frak m)$ is given by
$$\text{ad}_a(z)=\sum_{(a)} \{a_1, z\}S(a_2)=\{a,z\}$$ for all $a\in \frak m$ and $z\in \mathcal{S}(\frak m)$.
Thus it is clear that $\mathcal{S}(\frak m)^{\text{ad}}=\mathcal{Z}(\mathcal{S}(\frak m))$.


\section{$H$-action and Poisson primitive ideals of $\Bbb C[G]$}

Here we find the Poisson center of $\Bbb C[G]_w$, $w\in W\times W$, by modifying the statements and proofs in \cite[\S4.2]{HoLeT} using Poisson terminologies instead of those of noncommutative algebra, and prove that $\Bbb C[G]$ satisfies the Poisson Dixmier-Moeglin equivalence.

\subsection{}
Recall, as given in \ref{TORUCS}, that $H$ is a torus associated with the Cartan subalgebra $\frak h$ and that ${\bf L}$ is the  character group of $H$.  There is an action of $H$ on $\Bbb C[G]=\bigoplus_{(\lambda,\mu)\in{\bf L}\times{\bf L}}\Bbb C[G]_{\lambda,\mu}$ by
$$h\cdot z=\mu(h)z,\ \ \ h\in H, z\in \Bbb C[G]_{\lambda,\mu}.$$

\begin{lem}\label{HACTA}
(1) For all $h\in H$ and $a,z\in \Bbb C[G]$,
$$\text{ad}_a(h\cdot z)=h\cdot \text{ad}_a(z),$$
where $\text{ad}_a$ is the canonical  Poisson adjoint action of $\Bbb C[G]$.

(2) For each $w\in W\times W$,  $\Bbb C[G]_w$ is a Poisson module over $\Bbb C[G]$ with module structure
$$(z+I_w)a=za+I_w,\ \ \ a*(z+I_w)=\{a,z\}+I_w$$
for $a,z\in\Bbb C[G]$.

(3) For each $w\in W\times W$,  there is a Poisson adjoint action on $\Bbb C[G]_w$ defined by
$$\text{ad}_a(z+I_w)=\sum_{(a)}\{a_1,z\}S(a_2)+I_w,$$
that  satisfies
 $$\text{ad}_a(h\cdot(z+I_w))=h\cdot\text{ad}_a(z+I_w)$$
 for $a,z\in\Bbb C[G]$, $h\in H$.
 \end{lem}

\begin{proof}
(1) For $a=c_{f,v}^M\in \Bbb C[G]_{\lambda,\mu}$,   $z\in \Bbb C[G]_{\nu,\eta}$ and $h\in H$,
we have
$$\begin{array}{l}
\text{ad}_a(h\cdot z)=\eta(h)\text{ad}_a(z)=\eta(h)\sum \{c_{f,v_i}^M,z\}S(c^M_{g_i,v})\\
h\cdot \text{ad}_a(z)=h\cdot(\sum \{c_{f,v_i}^M,z\}S(c^M_{g_i,v}))=\eta(h)\sum \{c_{f,v_i}^M,z\}S(c^M_{g_i,v})
\end{array}$$
by (\ref{J}) and Theorem~\ref{PMSTRU}  since we may assume that   the dual bases $\{v_i\}$ and $\{g_i\}$ of $M$ and $M^*$ are weight vectors for all $i$.
It follows that
$\text{ad}_a(h\cdot z)=h\cdot \text{ad}_a(z)$
for all $h\in H$ and $a,z\in \Bbb C[G]$.

(2) Since $\Bbb C[G]$ is a Poisson module with the canonical Poisson module structure and $I_w$ is a Poisson ideal by Lemma~\ref{CPRIB}, $\Bbb C[G]/I_w$ is a Poisson module. Hence the localization $\Bbb C[G]_w$ is a Poisson module.

(3)
Since $I_w$ is a Poisson ideal by Lemma~\ref{CPRIB}, the canonical Poisson adjoint action acts on $\Bbb C[G]/I_w$.
Moreover the canonical Poisson adjoint action is a derivation by (\ref{ADJC}) and thus the canonical Poisson adjoint action
extends uniquely on the localization $\Bbb C[G]_w$.

Since $I_w$ is a homogeneous Poisson ideal by Lemma~\ref{CPRIB}, $\Bbb C[G]/I_w$ is also  $\bf L$-bigraded. Note that  there exists a generating set of $I_w$ consisting of $H$-eigenvectors and each element of
$\mathcal{E}_w$ is an $H$-eigenvector. Thus
 the action of $H$ on $\Bbb C[G]$ induces an action on $\Bbb C[G]_w$, and thus  the result
follows immediately by (1).
\end{proof}

\begin{notn}
Fix $w\in W\times W$. For $\Lambda\in {\bf L}^+, f\in V(\Lambda)^*$ and $v\in V(\Lambda)$, we set
$$z_f^+=c^{-1}_{w\Lambda}(c^{V(\Lambda)}_{f,v_\Lambda}+I_w),\ \ \
z_v^-=\tilde{c}^{-1}_{w\Lambda}(c^{V(\Lambda)^*}_{v,f_{-\Lambda}}+I_w).$$
Let $\{\omega_1,\ldots,\omega_n\}$ be a basis of ${\bf L}$ such that $\omega_i\in {\bf L}^+$ for all $i$. For
each $\lambda=\sum_i s_i\omega_i\in {\bf L}$, we define Poisson normal elements of $\Bbb C[G]_w$ by
$$c_{w\lambda}=\prod^n_{i=1}c^{s_i}_{w\omega_i},\ \ \tilde{c}_{w\lambda}=\prod^n_{i=1}\tilde{c}^{s_i}_{w\omega_i},
\ \ d_\lambda=(\tilde{c}_{w\lambda}c_{w\lambda})^{-1}.$$
Note that
\begin{equation}\label{ADDD}
c_{w(\lambda+\mu)}=c_{w\lambda}c_{w\mu},\ \ \tilde{c}_{w(\lambda+\mu)}=\tilde{c}_{w\lambda}\tilde{c}_{w\mu}
\end{equation}
for any $\lambda,\mu\in{\bf L}$.

Define subalgebras of $\Bbb C[G]_w$ by
$$\begin{array}{c}
C_w=\Bbb C[z^+_f,z^-_v, c_{w\lambda}\ |\ f\in V(\Lambda)^*,v\in V(\Lambda),\Lambda\in {\bf L}^+,\lambda\in {\bf L}]\\
C_w^+=\Bbb C[z^+_f\ |\ f\in V(\Lambda)^*,\Lambda\in {\bf L}^+ ], \ \ \
C_w^-=\Bbb C[z^-_v\ |\ v\in V(\Lambda),\Lambda\in {\bf L}^+].
\end{array}$$
\end{notn}

\begin{lem}\label{HACTB}
(1) For any $\Lambda,\Gamma\in {\bf L}^+$ and $f\in V(\Lambda)^*$, there exists an element
$g\in V(\Lambda+\Gamma)^*$ such that $z_f^+=z_g^+$.

(2) For any $\Lambda,\Gamma\in {\bf L}^+$ and $v\in V(\Lambda)$, there exists an element
$u\in V(\Lambda+\Gamma)$ such that $z_v^-=z_u^-$.
\end{lem}

\begin{proof}
It is proved by mimicking the proof of \cite[4.8]{HoLeT}.
\end{proof}

\begin{lem}\label{HACTC}
For each $w\in W\times W$, the algebras $C_w$ and $C_w^{\pm}$ are  $H$-stable Poisson subalgebras of $\Bbb C[G]_w$.
\end{lem}

\begin{proof}
Note that  $c_{w\Lambda}\in(\Bbb C[G]_w)_{-w_+\Lambda,\Lambda}$, $\widetilde{c}_{w\Lambda}\in(\Bbb C[G]_w)_{w_-\Lambda,-\Lambda}$ and thus
$$\begin{array}{cc}z_{f_{-\lambda}}^+\in(\Bbb C[G]_w)_{w_+\Lambda-\lambda,0},& z_{v_{\lambda}}^-\in(\Bbb C[G]_w)_{-w_-\Lambda+\lambda,0},\\
c_{w\lambda}\in(\Bbb C[G]_w)_{-w_+\lambda,\lambda},& \widetilde{c}_{w\lambda}\in(\Bbb C[G]_w)_{w_-\lambda,-\lambda}.
\end{array}$$
Hence
\begin{equation}\label{HACTD}
h\cdot z_f^+=z_f^+,\ \  h\cdot z_v^-=z_v^-, \ \ h\cdot c_{w\lambda}=\lambda(h)c_{w\lambda},
\ \  h\cdot \tilde{c}_{w\lambda}=\lambda(h)^{-1}\tilde{c}_{w\lambda}
\end{equation}
for all $h\in H$ and thus  the algebras $C_w$ and $C_w^{\pm}$ are $H$-stable.

Observe by (\ref{PMSTRU}) and Lemma~\ref{CPRIBB} that, for any
$\Lambda\in{\bf L}^+$, $f_{-\lambda}\in (V(\Lambda)^*)_{-\lambda}$ and $f_{-\mu}\in (V(\Lambda)^*)_{-\mu}$,
$$\begin{aligned}
\{z_{f_{-\lambda}}^+, z_{f_{-\mu}}^+\}&=az_{f_{-\lambda}}^+ z_{f_{-\mu}}^+-2\sum_{\alpha\in\bf R^+} z_{f_{-\lambda}x_\alpha}^+ z_{f_{-\mu}x_{-\alpha}}^+,
\end{aligned}$$
where $a=(\Phi_+\Lambda|\Lambda)-(\Phi_+\lambda|\mu)+(\Phi_+w_+\Lambda|\mu)-(\Phi_+w_+\Lambda| \lambda)\in\Bbb C$.
Hence $C_w^+$ is a Poisson subalgebra by Lemma~\ref{HACTB}.  Similarly, $C_w^-$ is  a Poisson subalgebra.

Observe that, for any $\Lambda\in{\bf L}^+$,
\begin{equation}\label{XXXX}
\{z^+_{f_{-\lambda}}, z^-_{v_\mu}\}=b z^+_{f_{-\lambda}} z^-_{v_\mu}+2\sum_{\alpha\in{\bf R}^+}z^+_{f_{-\lambda}x_\alpha} z^-_{x_{-\alpha}v_\mu}
\end{equation}
for some $b\in\Bbb C$ by (\ref{PMSTRU}) and Lemma~\ref{CPRIBB}. Thus $C_w$ is also a Poisson subalgebra by Lemma~\ref{HACTB}.
\end{proof}

\begin{thm}\label{HACTE}
(1) Let $C_w^H$ be the set of all fixed elements in $C_w$ under the action of $H$.
Then  $C_w^H=\Bbb C[z_f^+,z_v^-\ |\ f\in V(\Lambda)^*,v\in V(\Lambda),\Lambda\in {\bf L}^+]$, that is a Poisson subalgebra.

(2) The set $\mathcal{D}=\{d_\lambda\ |\ \lambda\in {\bf L}\}$ is a multiplicatively closed subset of $C_w^H$.
Moreover $\Bbb C[G]_w$ and $\Bbb C[G]_w^H$ are localizations of $C_w$ and $C_w^H$ at $\mathcal D$ respectively, that is,
$$\Bbb C[G]_w=(C_w)_{\mathcal{D}}, \ \ \ \Bbb C[G]_w^H=(C_w^H)_{\mathcal{D}},$$
where $\Bbb C[G]_w^H$ is the set of all fixed elements in $\Bbb C[G]_w$ under the action of $H$.

(3) For each $\lambda\in {\bf L}$, let $(\Bbb C[G]_w)_\lambda=\{a\in \Bbb C[G]_w\ |\ h\cdot a=\lambda(h)a\ \text{for all } h\in H\}$.
Then
$$\Bbb C[G]_w=\bigoplus_{\lambda\in {\bf L}}(\Bbb C[G]_w)_\lambda, \ \ \ (\Bbb C[G]_w)_\lambda=\Bbb C[G]_w^Hc_{w\lambda}.$$
Moreover each $(\Bbb C[G]_w)_\lambda$ is an invariant subspace of $\Bbb C[G]_w$ under the
Poisson adjoint action.
\end{thm}

\begin{proof}
(1) It is proved  by (\ref{HACTD}), (\ref{XXXX}) and Lemma~\ref{HACTC}.

(2) and   (3) These are proved by  mimicking the proof of \cite[Theorem 4.7]{HoLeT} using Poisson terminologies. We repeat it for completion.
 Let $\{v_i\}$ and $\{g_i\}$ be   dual bases of $V(\Lambda)$ and $V(\Lambda)^*$. Then
$$1=\epsilon(c_{f_{-\Lambda},v_\Lambda}^{V(\Lambda)})=\sum_iS(c_{f_{-\Lambda},v_i}^{V(\Lambda)})c_{g_i,v_\Lambda}^{V(\Lambda)}
=\sum_ic_{v_i,f_{-\Lambda}}^{V(\Lambda)^*}c_{g_i,v_\Lambda}^{V(\Lambda)}.
$$
Multiplying both sides of the equation by $d_\Lambda$ yields $d_\Lambda=\sum z_{v_i}^-z_{g_i}^+$.
Thus $\mathcal{D}\subseteq C_w^H$ by (1) and  $\mathcal{D}$ is a multiplicatively closed set by Lemma~\ref{HACTB}. Now, by Theorem~\ref{MMMM}, any element of $\Bbb C[G]_w$ is a sum of elements of the form
$c_{f,v_\Lambda}^{V(\Lambda)}c_{g,v_{-\Gamma}}^{V(\Gamma)}d_\Upsilon$, where $\Lambda,\Gamma, \Upsilon\in {\bf L}^+$.
This element lies in $(\Bbb C[G]_w)_\lambda$ if and only if $\Lambda-\Gamma=\lambda$.
In this case $c_{f,v_\Lambda}^{V(\Lambda)}c_{g,v_{-\Gamma}}^{V(\Gamma)}d_\Upsilon$ is equal to the element
$z_f^+z_g^-d_\Upsilon d_\Gamma^{-1}c_{w\lambda}\in [(C_w^H)_\mathcal{D}]c_{w\lambda}$.
Since $\Bbb C[G]_w$ is invariant under the Poisson adjoint action and the Poisson adjoint action on $\Bbb C[G]_w$ commutes with the action of $H$ by Lemma~\ref{HACTA}(3),
$(\Bbb C[G]_w)_\lambda$ is
invariant under the Poisson adjoint action.
(The fact that $(\Bbb C[G]_w)_\lambda$ is
invariant under the Poisson adjoint action is also proved by Lemma~\ref{HACTK} and Lemma~\ref{CPRIBB}.)
Hence the remaining assertions follow.
\end{proof}

\subsection{}
Let $N\in{\mathcal{C}}(\frak d)$. For $v\in N$ and  $g\in N^*$, define $\psi_{v\otimes g}\in \text{End}(N)$ by
$\psi_{v\otimes g}(z)=g(z)v$ for $z\in N$. Let us show that
$$\psi:N\otimes N^*\longrightarrow \text{End}(N),\ \ v\otimes g\mapsto\psi_{v\otimes g}$$
is an isomorphism of vector spaces.
Fix  bases $\{v_i\}_{i=1}^r$ and $\{g_i\}_{i=1}^r$ of $N$ and $N^*$ such that $g_i(v_j)=\delta_{ij}$. Every element
$w\in N\otimes N^*$ can be expressed by
$w=\sum_i w_i\otimes g_i$ for some $w_i\in N$. If $\psi_w=\psi(w)=0$ then
$$0=\psi_w(v_j)=\sum_ig_i(v_j)w_i=w_j$$
for all $j$ and thus $w=0$. It follows that $\psi$ is injective and surjective since the dimension of $N\otimes N^*$
is equal to that of $\text{End}(N)$.

Since $\psi_{[\sum_i v_i\otimes g_i]}(v_j)=v_j$ for all $j$, we have
\begin{equation}\label{IDENT}
\psi_{[\sum_i v_i\otimes g_i]}=\text{id}_N,
\end{equation}
where $\text{id}_N$ is the identity map on $N$.
Denote by $\zeta$ the canonical embedding
 \begin{equation}\label{HACTF}
 \zeta:\Bbb C\longrightarrow N\otimes N^*,\ \ \zeta(1)=\sum_{i=1}^r v_i\otimes g_i.
 \end{equation}

 Note that $\Bbb C$ and $N\otimes N^*$ are $U(\frak d)$-modules with module structures
 $$\begin{aligned}
 U(\frak d)\times\Bbb C&\longrightarrow\Bbb C, &&(x, 1)\mapsto x.1:=\epsilon(x)1\\
 U(\frak d)\times (N\otimes N^*)&\longrightarrow N\otimes N^*, &&(x, a\otimes b)\mapsto x.(a\otimes b):=\Delta(x)(a\otimes b)
 \end{aligned}$$
 since the counit map $\epsilon$ and the comultiplication map $\Delta$ in $U(\frak d)$ are algebra homomorphisms.

 \begin{lem}\label{HACTG}
 The map $\zeta$ is a homomorphism of $U(\frak d)$-modules.
 \end{lem}

\begin{proof}
Let  $x,y\in U(\frak d)$, $v\in N$ and  let $S(y)v=\sum_{k=1}^r a_kv_k$, $a_k\in\Bbb C$.
Then
$$\psi_{[\sum_i(xv_i)\otimes(yg_i)]}(v)=\sum_ig_i(S(y)v)xv_i=\sum_{i,k}a_kg_i(v_k)xv_i=x\sum_ka_kv_k=xS(y)v.$$
Hence, for $z\in U(\frak d)$,
$$\begin{aligned}
\psi_{z\zeta(1)}=\psi_{z\sum_{i}v_i\otimes g_i}&=\psi_{[\sum_{i,(z)}(z_1v_i)\otimes(z_2g_i)]}\\
&=\sum_{(z)}z_1S(z_2)\text{id}_N=\epsilon(z)\text{id}_N\\
&=\epsilon(z)\psi_{[\sum_iv_i\otimes g_i]}=\psi_{\epsilon(z)\zeta(1)}
\end{aligned}$$
by (\ref{IDENT}) and (\ref{HACTF}).
Since $\psi$ is an isomorphism, we have that
$z\zeta(1)=\epsilon(z)\zeta(1)=\zeta(\epsilon(z)1)$
for every $z\in U(\frak d)$.
\end{proof}

\begin{lem}\label{HACTH}
Let  $ N\in{\mathcal{C}}(\frak d)$,
$g\in N^*$, $v\in N$. Let $\{v_i\}_{i=1}^r$ and $\{g_i\}_{i=1}^r$ be bases of $N$ and $N^*$ such that $g_i(v_j)=\delta_{ij}$.
Then, for any $c_{f,p}=c_{f,p}^N\in\Bbb C[G]$,
$$c_{(1\otimes\zeta)^*(f\otimes g\otimes v),p}=a c_{f,p}$$ where $a=\sum_{i}g(v_i)g_i(v)\in\Bbb C$.
\end{lem}

\begin{proof}
For any $z\in U(\frak d)$,
$$\begin{aligned}
c_{(1\otimes\zeta)^*(f\otimes g\otimes v),p}(z)
&=(1\otimes\zeta)^*(f\otimes g\otimes v)(zp)=(f\otimes g\otimes v)(1\otimes\zeta)(zp)\\
&=(f\otimes g\otimes v)(\sum_i zp\otimes v_i\otimes g_i)\\
&=\sum_i f(zp)g(v_i)g_i(v)=ac_{f,p}(z).
\end{aligned}
$$
 Hence $c_{(1\otimes\zeta)^*(f\otimes g\otimes v),p}=a c_{f,p}$.
\end{proof}

A $\Bbb C[G]$-module $M$ is said to be {\it locally closed} if, for each $z\in M$, $\Bbb C[G]z$ is finite dimensional.

\begin{lem}\label{HACTK}
Let $\Lambda\in{\bf L}^+$, $f\in (V(\Lambda)^*)_{-\lambda}$, $p\in V(\Lambda)_{\lambda}$,  $N\in{\mathcal{C}}(\frak d)$.
Let $\{v_i\}_{i=1}^r$ and $\{g_i\}_{i=1}^r$ be bases of $N$ and $N^*$ such that $g_i(v_j)=\delta_{ij}$.
Then, for $c=c_{g,v}^N\in\Bbb C[G]_{-\eta,\gamma}$,

(1)  $\text{ad}_c(z_f^+)=a_0z_{f}^+ +\sum_{\alpha\in\bf R^+} a_\alpha z_{fx_\alpha}^+$, where
$$\begin{aligned}
a_0&=[(\Phi_+\lambda|\eta)-(\Phi_+w_+\Lambda|\eta)]\sum_ig(v_i)g_i(v),\\
 a_\alpha&=2\sum_i(gx_{-\alpha})(v_i)g_i(v).
  \end{aligned}$$
  In particular, if $\eta\neq\gamma$ then $a_0=0$ and if $\eta=\gamma$ then $a_\alpha=0$ for all $\alpha\in\bf R^+$.

(2)  $\text{ad}_c(z_p^-)=b_0z_{p}^- + \sum_{\alpha\in{\bf R}^+} b_\alpha z_{x_{-\alpha}p}^-$, where
$$\begin{aligned}
b_0&=[(\Phi_-w_-\Lambda|\eta)-(\Phi_-\lambda|\eta)]\sum_ig(v_i)g_i(v),\\
 b_\alpha&=2\sum_i(gx_\alpha)(v_i)g_i(v).
  \end{aligned}$$
 In particular, if $\eta\neq\gamma$ then $b_0=0$ and if $\eta=\gamma$ then $b_\alpha=0$ for all $\alpha\in\bf R^+$.

(3) If $a_\alpha\neq 0$ for some $\alpha\in\bf R^+$, then $\text{ad}_c(z_p^-)=0$. Conversely, if  $b_\alpha\neq 0$ for some $\alpha\in\bf R^+$, then $\text{ad}_c(z_f^+)=0$.
In particular, if $\eta>\gamma$ then $\text{ad}_c(z_f^+)=0$ and if $\eta<\gamma$ then $\text{ad}_c(z_p^-)=0$.

(4) $\text{ad}_c(C_w^+)\subseteq C_w^+$, $\text{ad}_c(C_w^-)\subseteq C_w^-$,
 $\text{ad}_c(C_w^H)\subseteq C_w^H$ for any $c\in \Bbb C[G]$ and
the Poisson adjoint actions on $C_w^{\pm}$ and $C_w^H$ are locally finite.
\end{lem}

\begin{proof}
For convenience, we will write $z$ for $z+I_w\in\Bbb C[G]_w$. We may assume that all $v_i$ and $g_i$ are weight vectors with
weights $\gamma_i$ and $-\gamma_i$ respectively.

(1)
We have
$$\begin{aligned}
\text{ad}_c(z_f^+)&=\sum_i \{c_{g,v_i}, c_{w\Lambda}^{-1}c_{f,v_\Lambda}\}S(c_{g_i,v})\\
&=\sum_i(-c_{w\Lambda}^{-1}\{c_{f,v_\Lambda},c_{g,v_i}\}+c_{w\Lambda}^{-2}c_{f,v_\Lambda}\{c_{w\Lambda},c_{g,v_i}\})c_{v,g_i}\\
&=\sum_i[(\Phi_+\lambda|\eta)-(\Phi_+w_+\Lambda|\eta)]c_{w\Lambda}^{-1}c_{f\otimes g\otimes v, v_\Lambda\otimes v_i\otimes g_i}
+2\sum_{\alpha\in{\bf R}^+, i}c_{w\Lambda}^{-1}c_{fx_\alpha\otimes gx_{-\alpha}\otimes v, v_\Lambda\otimes v_i\otimes g_i}\\
&\qquad\qquad(\text{by (\ref{MPBB}) and Lemma~\ref{CPRIBB})}\\
&=[(\Phi_+\lambda|\eta)-(\Phi_+w_+\Lambda|\eta)]c_{w\Lambda}^{-1}c_{f\otimes g\otimes v, (1\otimes\zeta)(v_\Lambda)}
+2\sum_{\alpha\in\bf R^+}c_{w\Lambda}^{-1}c_{fx_\alpha\otimes gx_{-\alpha}\otimes v, (1\otimes\zeta)(v_\Lambda)}\\
&\qquad\qquad(\text{by (\ref{HACTF}))}\\
&=[(\Phi_+\lambda|\eta)-(\Phi_+w_+\Lambda|\eta)]c_{w\Lambda}^{-1}c_{(1\otimes\zeta)^*(f\otimes g\otimes v), v_\Lambda}
+2\sum_{\alpha\in\bf R^+}c_{w\Lambda}^{-1}c_{(1\otimes\zeta)^*(fx_\alpha\otimes gx_{-\alpha}\otimes v), v_\Lambda}\\
&\qquad\qquad(\text{by Lemma~\ref{HACTG})}\\
&=a_0z_f^++\sum_{\alpha\in\bf R^+}a_\alpha z_{fx_\alpha}^+. \ \ \ (\text{by Lemma~\ref{HACTH}})
\end{aligned}$$
 If $\eta\neq\gamma$ then $g(v_i)g_i(v)=0$ for each $1\leq i\leq r$ by Lemma~\ref{WEIG}, thus  $a_0=0$.
 If $\eta=\gamma$ then $gx_{-\alpha}(v_i)g_i(v)=0$ for each $1\leq i\leq r$ by Lemma~\ref{WEIG}, thus $a_\alpha=0$ for each $\alpha\in{\bf R}^+$.

(2) Similar to (1).

(3)
Let $z$ be a weight vector with weight $\mu$. Then we  write $\mu=\text{wt}(z)$.
Observe that
\begin{equation}\label{HACTL}
\begin{aligned}
(gx_{-\alpha})(v_i)g_i(v)\neq0&\Rightarrow \text{wt}(gx_{-\alpha})=-\text{wt}(v_i), \text{wt}(g_i)=-\text{wt}(v)\\
&\Rightarrow
\eta+\alpha=\gamma
\end{aligned}\end{equation}
and
\begin{equation}\label{HACTM}
\begin{aligned}
(gx_\alpha)(v_i)g_i(v)\neq0&\Rightarrow \text{wt}(gx_\alpha)=-\text{wt}(v_i), \text{wt}(g_i)=-\text{wt}(v)\\
&\Rightarrow
\eta-\alpha=\gamma
\end{aligned}\end{equation}
by Lemma~\ref{WEIG}.

If $a_\alpha\neq 0$ for some $\alpha\in\bf R^+$ then $gx_{-\alpha}(v_i)g_i(v)\neq0$ for some $i$. Thus
$\eta\neq\gamma$ and $(gx_\alpha)(v_i)g_i(v)=0$ for all $i$ and all $\alpha\in{\bf R}^+$ by (\ref{HACTL}) and (\ref{HACTM}).  Hence $b_0=0$ and $b_\alpha=0$  for all $\alpha\in{\bf R}^+$ by (2), and thus $\text{ad}_c(z_p^-)=0$. In particular, if $\eta>\gamma$ then $\text{ad}_c(z_f^+)=0$ by (\ref{HACTL}).

Similarly, if  $b_\alpha\neq 0$ for some $\alpha\in\bf R^+$, then $\text{ad}_c(z_f^+)=0$, and if $\eta<\gamma$ then $\text{ad}_c(z_p^-)=0$ by (\ref{HACTM}).

(4)
It is  clear by (1), (2) and Theorem~\ref{HACTE} since every object of $\mathcal{C}(\frak d)$  is finite dimensional.
\end{proof}

\subsection{}
Henceforth we denote by $U$ the Poisson enveloping algebra of $\Bbb C[G]$.
(See \cite[Definition 3]{Oh5} for the definition of Poisson enveloping algebra.)
Note that  $U$ is an associative algebra. Since  $\Bbb C[G]_w$ is
a Poisson module over $\Bbb C[G]$ by Lemma~\ref{HACTA}(2), $C_w^+$ and $C_w^-$ are Poisson modules over $\Bbb C[G]$ with module structure
$$za=\epsilon(a)z,\ \ \ a*z=\text{ad}_a z$$
for $a\in\Bbb C[G], z\in C_w^\pm$, by Lemma~\ref{ADJD}(4) and Lemma~\ref{HACTK}(1), (2). Thus $C_w^\pm$ are $U$-modules by \cite[Corollary 6]{Oh5}.
Recall that the socle  of a module $M$,
denoted by $\text{Soc}(M)$,  is the sum of all minimal submodules of $M$.

\begin{lem}\label{HACTN}
$\text{Soc}(C_w^+)=\Bbb C$ and $\text{Soc}(C_w^-)=\Bbb C$.
\end{lem}

\begin{proof}
For $\Lambda\in{\bf L}^+$, let $T_\Lambda=\{z_f^+ \ |\ f\in V(\Lambda)^*\}$.
Then $C_w^+=\cup_{\Lambda\in{\bf L}^+} T_\Lambda$ by Lemma~\ref{HACTB}.
Let $M$ be a minimal submodule of $C_w^+$ and choose $0\neq z_f^+\in M$.
Applying $\text{ad}_c$ on $z_f^+$ for a suitable element $c=c_{g,v}\in \Bbb C[G]$, we may assume that $f\in V(\Lambda)^*_{-w_+\Lambda}$ by Lemma~\ref{HACTK}(1).
Hence $1=z_f^+\in M$, and thus $\Bbb C= Uz_{f}^+=M$ since  $\text{ad}_c(1)=0$ for all $c\in \Bbb C[G]$ by (\ref{ADJC}).
 It follows that
 $\text{Soc}(C_w^+)=\Bbb C$.
Analogously  we have $\text{Soc}(C_w^-)=\Bbb C$.
 \end{proof}

\begin{thm}\label{HACTO}
There is no nontrivial Poisson ideal $I$ of $C_w^H$ such that $\text{ad}_c(I)\subseteq I$ for all $c\in \Bbb C[G]$.
\end{thm}

\begin{proof}
Observe that  $C_w^+\otimes C_w^-$ is a Poisson module over $\Bbb C[G]$ with module structure
$$\begin{aligned}
&(C_w^+\otimes C_w^-)\times \Bbb C[G]\longrightarrow C_w^+\otimes C_w^-,\ \ \ (z\otimes y, a)\mapsto (z\otimes y)\cdot a=\epsilon(a)(z\otimes y)\\
&\Bbb C[G]\times (C_w^+\otimes C_w^-)\longrightarrow C_w^+\otimes C_w^-, \ \ \
(a, z\otimes y)\mapsto a*(z\otimes y)=\text{ad}_a(z)\otimes y+x\otimes\text{ad}_a(y)
\end{aligned}$$
by Lemma~\ref{ADJD} and that, by Lemma~\ref{HACTK}, $C_w^H$ is also a Poisson module over $\Bbb C[G]$ with module structure
$$za=\epsilon(a)z,\ \ \ a*z=\text{ad}_a z$$
for $a\in\Bbb C[G], z\in C_w^H$.
Let $c=c_{f_{-\lambda},v_\mu}\in\Bbb C[G]_{-\lambda,\mu}, \lambda\neq\mu$.
We verify that $c*z=0$ for any $z\in \text{Soc}(C_w^+\otimes C_w^-)$.
 Express $z\in \text{Soc}(C_w^+\otimes C_w^-)$ by  $z=\sum_i z_{g_i}^+\otimes z_{v_i}^-$, where $g_i$ and $v_i$ are weight vectors.
By Lemma~\ref{HACTB} and Lemma~\ref{HACTK}, we may assume that   weights of $g_i$ and $v_i$ are $-\ w_+\Lambda$ and $w_-\Lambda$, $\Lambda\in{\bf L}^+$,
respectively,
 by considering the elements
$c_{g_{-\eta},v_\gamma}*z\in Uz$, $\eta\neq\gamma$.
Thus $c*z=0$ since both  $\text{ad}_c(z_{f_i}^+)=0$  and  $\text{ad}_c(z_{v_i}^-)=0$ by Lemma~\ref{HACTK}(1), (2).

Let us show that $\text{Soc}(C_w^+\otimes C_w^-)=\Bbb C$.
Every element of $C_w^+\otimes C_w^-$ may be written by $\sum_i a_i\otimes b_i$, where the $b_i$ are linearly independent elements of the form $z^-_{v_\mu}$, $\mu\in\bf L$.
Suppose that $\sum_i a_i\otimes b_i\in \text{Soc}(C_w^+\otimes C_w^-)$.
Replacing each $a_i$ by $\sum_j z^+_{g_{\mu_{ij}}}$ and acting $c_{f_{-\lambda},v_\lambda}$ on $\sum_i a_i\otimes b_i$,
we may assume that each $a_i$ is a common  eigenvector of  all $c_{f_{-\lambda},v_\lambda}$, $\lambda\in{\bf L}$,  by Lemma~\ref{HACTK}(1), (2).
For any $c_{f_{-\lambda},v_\mu}$ such that $\lambda<\mu$,
$\text{ad}_{c_{f_{-\lambda},v_\mu}}(b_i)=0$ by Lemma~\ref{HACTK}(3). Thus
$$0=c_{f_{-\lambda},v_\mu}*(\sum_i a_i\otimes b_i)=\sum_i\text{ad}_{c_{f_{-\lambda},v_\mu}}(a_i)\otimes b_i.$$
 It follows that $\text{ad}_{c_{f_{-\lambda},v_\mu}}(a_i)=0$ for all $\lambda\neq\mu $ by Lemma~\ref{HACTK}(3),
  and thus $Ua_i=\Bbb Ca_i$ by Lemma~\ref{HACTK}(1). Hence $a_i\in\text{Soc}(C_w^+)=\Bbb C$  by Lemma~\ref{HACTN}. Therefore $\sum a_i\otimes b_i\in\text{Soc}(\Bbb C\otimes C_w^-)=\Bbb C\otimes\Bbb C$, as claimed.

The multiplication map $\sigma: C_w^+\otimes C_w^-\longrightarrow C_w^H$ is a $U$-module epimorphism by Theorem~\ref{HACTE}(1) since the canonical Poisson adjoint action
is a derivation by (\ref{ADJC}). If  $I$ is a nonzero Poisson ideal of $C_w^H$
invariant under the Poisson adjoint action then
$I$ is a submodule of $C_w^H$ and thus $\sigma^{-1}(I)$ contains a minimal submodule since $C_w^+\otimes C_w^-$
is locally finite by Lemma~\ref{HACTK}(4).  Therefore
$\sigma^{-1}(I)$ has a nonzero element of $\text{Soc}(C_w^+\otimes C_w^-)=\Bbb C$. It follows that   $I=C_w^H$.
\end{proof}

\begin{cor}\label{HACTP}
The Poisson algebra $\Bbb C[G]_w^H$ has no nontrivial Poisson ideal $I$ such that
$\text{ad}_c(I)\subseteq I$ for all $c\in \Bbb C[G]$.
Moreover $(\Bbb C[G]_w^H)^{\text{ad}}=\Bbb C$.
\end{cor}

\begin{proof}
Since $\Bbb C[G]_w^H$ is a localization of $C_w^H$ by  Theorem~\ref{HACTE}(2),
$\Bbb C[G]_w^H$ has no nontrivial Poisson ideal invariant under the Poisson adjoint action by Theorem~\ref{HACTO}.
Suppose that there exists an element $y\in (\Bbb C[G]_w^H)^{\text{ad}}\setminus\Bbb C$.
Then $y$ is transcendental over $\Bbb C$ and the ideal $(y-a)\Bbb C[G]_w^H$, $a\in \Bbb C$, is a nonzero Poisson ideal of
$\Bbb C[G]_w^H$ invariant under the Poisson adjoint action by Theorem~\ref{ADJK}.
Thus $y-a$ is invertible for each $a\in\Bbb C$, that is a contradiction
since  $\{(y-a)^{-1}|a\in \Bbb C\}$ is an uncountably infinite and linearly independent set but
$\Bbb C[G]_w^H$ has a countable dimension over $\Bbb C$.
\end{proof}

\begin{thm}\label{HACTQ}
Fix $w\in W\times W$. Let $\mathcal{Z}_w$ be the Poisson center of $\Bbb C[G]_w$, that is,
$$\mathcal{Z}_w=\{a\in \Bbb C[G]_w\ |\ \{a, b\}=0\ \text{for all}\ b\in \Bbb C[G]_w\}.$$ Then

(1) $\mathcal{Z}_w=\Bbb C[G]_w^{\text{ad}}$.

(2) $\mathcal{Z}_w=\oplus_{\lambda\in {\bf L}} \mathcal{Z}_\lambda$, where $\mathcal{Z}_\lambda=\mathcal{Z}_w\cap \Bbb C[G]_w^Hc_{w\lambda}$.

(3) If $\mathcal{Z}_\lambda\neq0$ then $\mathcal{Z}_\lambda=\Bbb Cu_\lambda$ for some unit $u_\lambda$.

(4) $\mathcal{Z}_w$ is isomorphic to a group algebra of a free abelian group with finite rank over $\Bbb C$.

(5) The group $H$ acts transitively on the maximal ideals of $\mathcal{Z}_w$.
\end{thm}

\begin{proof}
(1) It is proved by Theorem~\ref{ADJK}.

(2) It is proved by Theorem~\ref{HACTE}(3).

(3) Let $u_\lambda$ be a nonzero element of $\mathcal{Z}_\lambda$. Then $u_\lambda=ac_{w\lambda}$ for some $a\in\Bbb C[G]_w^H$.
This implies that $a=u_\lambda c_{w\lambda}^{-1}$ is a Poisson normal element by Lemma~\ref{CPRIBB}. Hence the ideal of $\Bbb C[G]_w^H$ generated by $a$
is a nonzero Poisson ideal invariant under the Poisson adjoint action. It follows that $a$  is a unit by Corollary~\ref{HACTP}
and thus $u_\lambda$ is a unit.
If $z\in \mathcal{Z}_\lambda$ then $zu_\lambda^{-1}\in \mathcal{Z}_0$. Thus  $\mathcal{Z}_\lambda=\Bbb Cu_\lambda$ since $\mathcal{Z}_0=\Bbb C$ by Corollary~\ref{HACTP}.

(4) Let ${\bf M}=\{\lambda\in{\bf L}|\mathcal{Z}_\lambda\neq0\}$.  Then ${\bf M}$ is a subgroup of ${\bf L}$ by (2), (3) and (\ref{ADDD}). Thus $\mathcal{Z}_w$ is isomorphic to the group algebra of the free abelian group ${\bf M}$ of finite rank over $\Bbb C$ by (2), (3).

(5) By the Hilbert's Nullstellensatz and (4), $H$ acts transitively on the maximal ideals of $\mathcal{Z}_w$.
\end{proof}

\begin{thm}\label{HACTR}
For each $w\in W\times W$, the Poisson ideals of $\Bbb C[G]_w$ are generated by their intersection with the Poisson center $\mathcal{Z}_w$.
\end{thm}

\begin{proof}
Note that $\Bbb C[G]_w$ and $\Bbb C[G]_w^H$ are Poisson modules over $\Bbb C[G]$ with module structure
$$za=\epsilon(a)z,\ \ \ a*z=\text{ad}_a z$$
for $a\in \Bbb C[G]$ by Lemma~\ref{HACTA}(2), Lemma~\ref{ADJD}(4), Theorem~\ref{HACTE}(1) and Lemma~\ref{HACTK}(1), (2).
The following argument is a modification of  \cite[Proof of Theorem 4.15]{HoLeT}.
Any element $f\in\Bbb C[G]_w$ can be written uniquely in the form $f=\sum a_\lambda c_{w\lambda}$ by Theorem~\ref{HACTE}(3),
where $a_\lambda\in \Bbb C[G]_w^H$.
Define $\pi:\Bbb C[G]_w\longrightarrow \Bbb C[G]_w^H$ to be the projection given by
$\pi(\sum a_\lambda c_{w\lambda})=a_0$. Since $\text{ad}_c(a_\lambda c_{w\lambda})\in (\Bbb C[G]_w)_\lambda$ for $c\in  \Bbb C[G]$ by Theorem~\ref{HACTE}(3),
$$\pi(\text{ad}_c(f))=\text{ad}_c(a_0)=\text{ad}_c\pi(f)$$
for $c\in  \Bbb C[G]$. Thus $\pi$ is  a homomorphism of Poisson modules. Define the support of $f$ to be
$\text{Supp}(f)=\{\lambda\in{\bf L}\ |\ a_\lambda\neq0\}$. Let $I$ be a Poisson ideal of $\Bbb C[G]_w$.
For any set $Y\subseteq{\bf L}$ such that $0\in Y$, define
$$I_Y=\{b\in \Bbb C[G]_w^H\ |\ b=\pi(f)\text{ for some }f\in I\text{ such that }\text{Supp}(f)\subseteq Y\}.$$
Since $I$ is a Poisson ideal of $\Bbb C[G]_w$, $I$ is invariant under the Poisson adjoint action.
Let $b\in I_Y$. Then $b=\pi(f)$ for some $f=\sum a_\lambda c_{w\lambda}\in I$ such that $\text{Supp}(f)\subseteq Y$.
For any $a\in\Bbb C[G]^H_w$,
$$\begin{aligned}
I\ni af&=\sum (aa_\lambda) c_{w\lambda},\\
I\ni\{f,a\}&=\sum(\{a_\lambda,a\}c_{w\lambda}+a_\lambda\{a,c_{w\lambda}\})=\sum(\{a_\lambda,a\}c_{w\lambda}+caa_\lambda c_{w\lambda}), \ \ c\in\Bbb C
\end{aligned}$$
by Lemma~\ref{CPRIBB} and
thus $\text{Supp}(\{f,a\})\subseteq\text{Supp}(f)\subseteq Y$  and $\{b,a\}\in I_Y$. Hence
 $I_Y$ is also a Poisson ideal of $\Bbb C[G]_w^H$
invariant under the Poisson adjoint action since   $\pi$ is a Poisson module homomorphism and $\text{Supp}(\text{ad}_c(f))\subseteq Y$ for $\text{Supp}(f)\subseteq Y$ and $c\in \Bbb C[G]$. Hence
 $I_Y$ is either zero or $\Bbb C[G]_w^H$ for each $Y\subseteq {\bf L}$ by Corollary~\ref{HACTP}.

Now let $I'=(I\cap \mathcal{Z}_w)\Bbb C[G]_w$ and suppose $I\neq I'$. Choose an element $f=\sum a_\lambda c_{w\lambda}\in I\setminus I'$
whose support $S$ has the smallest cardinality. We may assume without loss of generality that $0\in S$.
Suppose that there exists $g\in I'$
with $\text{Supp}(g)\subseteq S$ and fix $\lambda\in \text{Supp}(g)$. Then $gc_{w\lambda}^{-1}\in I'$ and $0\in \text{Supp}(gc_{w\lambda}^{-1})$.
Thus there exists an element $g'\in I'$ with  $\text{Supp}(g')\subseteq \text{Supp}(gc_{w\lambda}^{-1})$ and $\pi(g')=1$ by the above paragraph.
But then $f-a_\lambda g'c_{w\lambda}$ is an element of $I\setminus  I'$ with smaller support than $f$.
Thus there can be no elements in $ I'$ whose support is contained in $S$.

Since $0\neq a_0\in I_S$, $I_S=\Bbb C[G]_w^H$ by the first paragraph.
Hence we may assume that $\pi(f)=a_0=1$.
Then $\text{ad}_c(f)\in I'$ for any $c\in\Bbb C[G]$ since $|\text{Supp}(\text{ad}_c(f))|<|\text{Supp}(f)|$
and $\text{ad}_c(f)\in I$, thus $\text{ad}_c(f)=0$ for any $c\in\Bbb C[G]$ by
the second paragraph since
$\text{Supp}(\text{ad}_c(f))\subseteq S$.
It follows that $f\in I\cap \Bbb C[G]_w^{\text{ad}}=I\cap \mathcal{Z}_w\subseteq I'$ by Theorem~\ref{HACTQ}(1), that is  a contradiction.
This completes the proof.
\end{proof}

\subsection{}
Recall the  Poisson Dixmier-Moeglin equivalence in \cite[Theorem 2.4 and preceding comment]{Oh4}.
Let ${\bf k}$ be an algebraically closed field with characteristic zero. A Poisson ${\bf k}$-algebra $A$
is said to satisfy the  {\it Poisson Dixmier-Moeglin equivalence} if the
following conditions are equivalent: For a Poisson prime ideal $P$
of $A$,
\begin{itemize}
\item[(i)] $P$ is Poisson primitive (i.e., there exists a maximal ideal $M$ of $A$ such that $P$ is the largest
Poisson ideal contained in $M$).
\item[(ii)] $P$ is rational (i.e., the Poisson center of the quotient field of $A/P$ is equal to ${\bf k}$).
\item[(iii)] $P$ is locally closed (i.e., the intersection of all   Poisson prime ideals properly containing $P$
is strictly larger than $P$).
\end{itemize}

Note that ${\bf L}$ is the character group of the torus $H$. Hence there is an action of $H\times H$ on $\Bbb C[G]$ defined by
\begin{equation}\label{MACT}
(H\times H)\times \Bbb C[G]\longrightarrow\Bbb C[G],\ \ (h,h').c^M_{f,v}=\lambda(h)\mu(h')c^M_{f,v},
\end{equation}
where $f\in(M^*)_\lambda,v\in M_\mu$. Since $(\lambda+\mu)(h)=\lambda(h)\mu(h)$ for all $\lambda,\mu\in{\bf L}$ and $h\in H$, each element  $(h,h')\in H\times H$ acts by Poisson automorphism by Theorem~\ref{PMSTRU}.
Let $\{\omega_1,\ldots,\omega_n\}$ be a basis of ${\bf L}$ such that ${\bf L}^+=\sum_i\Bbb Z_{\geq0}\omega_i$, as in the proof of Theorem~\ref{ALGEB}. Then $H\times H$ acts rationally on $\Bbb C[G]$ since $\Bbb C[G]$ is generated by finitely many eigenvectors of the form  $c_{f_{-\lambda},v_\mu}^{V(\omega_i)}$, $v_\mu\in V(\omega_i)_\mu, f_{-\lambda}\in (V(\omega_i)^*)_{-\lambda}$, and each $V(\omega_i)$ is finite dimensional. (See the proof of Theorem~\ref{ALGEB}.)

An ideal $I$ of $\Bbb C[G]$ is said to be {\it $H\times H$-ideal} (or $H\times H$-stable) if $(H\times H)(I)\subseteq I$. An $H\times H$-ideal $P$ is said to be {\it $H\times H$-prime} if, for $H\times H$-ideals $I$ and $J$, $IJ\subseteq P$ implies $I\subseteq P$ or $J\subseteq P$. A Poisson ideal $I$ of $\Bbb C[G]$ is said to be {\it Poisson  $H\times H$-prime ideal} if $I$ is $H\times H$-prime ideal.
For an ideal $J$ of $\Bbb C[G]$, we denote $$(J:H\times H)=\bigcap_{(h,h')\in H\times H}(h,h')(J).$$
Note that $(J:H\times H)$ is an $H\times H$-ideal and that $(J:H\times H)$ is $H\times H$-prime if $J$ is prime.

\begin{lem}\label{FINL}
Let $a\in\Bbb C[G]_{\lambda,\mu}$ and $a'\in\Bbb C[G]_{\lambda',\mu'}$. If $(\lambda,\mu)\neq(\lambda',\mu')$ then there exists
an element $(h,h')\in H\times H$ such that the eigenvalues of $a$ and $a'$ with respect to the action of $(h,h')$ are distinct.
\end{lem}

\begin{proof}
Since $(\lambda,\mu)\neq(\lambda',\mu')$, we have that $\lambda\neq\lambda'$ or $\mu\neq\mu'$, say $\lambda\neq\lambda'$. Then $\lambda$ and $\lambda'$ are linearly independent by \cite[Lemma of \S16.1 ]{Hum}. Choose any $h'\in H$.
By \cite[Lemma C of \S16.2 ]{Hum}, there exists an element $h\in H$ such that $\lambda(h)\neq \lambda'(h)\mu(h')^{-1}\mu'(h')$. Hence $(h,h').a=\lambda(h)\mu(h')a$, $(h,h').a'=\lambda'(h)\mu'(h')a'$ and $\lambda(h)\mu(h')\neq \lambda'(h)\mu'(h')$.
\end{proof}

Let $w\in W\times W$. Since $I_w$ is homogeneous by Lemma~\ref{CPRIB} and  all elements of the multiplicatively closed set $\mathcal{E}_w$ are
$H\times H$-eigenvectors,  (\ref{MACT}) induces an action of $H\times H$ on $\Bbb C[G]_w$.

\begin{lem}\label{FINLL}
Every Poisson $H\times H$-prime ideal of $\Bbb C[G]_w$ is zero. In particular $I_w$ is an $H\times H$-stable Poisson prime ideal of $\Bbb C[G]$.
\end{lem}

\begin{proof}
Let $Q$ be a Poisson $H\times H$-prime ideal of $\Bbb C[G]_w$. Then  $Q$  is a Poisson prime ideal by \cite[II.1.12. Corollary]{BrGo}.
Suppose that $Q\neq0$.  Then there exists a nonzero element $a\in Q\cap\mathcal{Z}_w$ by Theorem~\ref{HACTR}. Acting $H\times H$ on $a$, we may assume that  $a=c^{V(\Lambda)}_{f_{-\lambda},v_\Lambda}$ or
$a=c^{V(\Lambda)^*}_{v_{\lambda},f_{-\Lambda}}$ for some $\lambda$ and $\Lambda$ by Theorem~\ref{HACTE} and Lemma~\ref{FINL}, say $a=c^{V(\Lambda)}_{f_{-\lambda},v_\Lambda}\in Q$. Applying Poisson adjoint action on $a$, we have that $Q$ contains an element $c_{w\Lambda}$ by Lemma~\ref{HACTK}(1), which is a contradiction since $c_{w\Lambda}$ is invertible. Hence $Q=0$.

Let $Q$ be a minimal prime ideal of $\Bbb C[G]_w$. Then $Q$ is a Poisson ideal by \cite[Lemma 1.1(c)]{Good3} and the $H\times H$-prime ideal $(Q:H\times H)$ is prime by \cite[II.1.12. Corollary]{BrGo}.
Hence $Q=(Q:H\times H)$, that is a Poisson $H\times H$-prime ideal of $\Bbb C[G]_w$. Therefore $Q=0$ by the above paragraph. It follows that $I_w$ is $H\times H$-stable Poisson prime.
\end{proof}

\begin{cor}\label{FF}
The Poisson $H\times H$-prime ideals of $\Bbb C[G]$ are only the ideals $I_w$, $w\in W\times W$.
\end{cor}

\begin{proof}
It follow by Lemma~\ref{FINLL}.
\end{proof}

\begin{thm}
(1) The Poisson algebra $\Bbb C[G]$ satisfies the Poisson Dixmier-Moeglin equivalence.
More precisely,
$$\aligned
\Pprim \Bbb C[G]&=\{\text{locally closed  Poisson prime ideals}\}\\
&=\{\text{rational  Poisson prime ideals}\}\\
&=\bigsqcup_{w\in W\times W}\ \{\text{maximal elements of
}\Pspec_w \Bbb C[G]\}
\endaligned$$

(2) $\Pspec \Bbb C[G]$ and $\Pprim \Bbb C[G]$ are topological quotients of
$\spec \Bbb C[G]$ and $\max \Bbb C[G]$ respectively.
\end{thm}

\begin{proof}
(1) It follows by Theorem~\ref{ALGEB}, Corollary~\ref{PPPP}, Corollary~\ref{FF} and \cite[Theorem 4.3]{Good3}.

(2) It follows by  \cite[Theorem 4.1]{Good3}.
\end{proof}


\bibliographystyle{amsplain}


\providecommand{\bysame}{\leavevmode\hbox to3em{\hrulefill}\thinspace}
\providecommand{\MR}{\relax\ifhmode\unskip\space\fi MR }
\providecommand{\MRhref}[2]{%
  \href{http://www.ams.org/mathscinet-getitem?mr=#1}{#2}
}
\providecommand{\href}[2]{#2}

\end{document}